%% file: permuting_matrix_entries.tex
\documentclass{amsart}[12pt]
\usepackage[paperheight=11in,paperwidth=8.5in,left=1in,top=1in,right=1in,bottom=1in]{geometry}
\usepackage[pdfstartpage={},pdfstartview={FitH},pdftitle={},pdfauthor={},pdfsubject={},pdfcreator={},pdfproducer={},pdfkeywords={},pagebackref=false,colorlinks=true,linkcolor={red!87.5!black},citecolor={blue!50!black},linktocpage=true]{hyperref}

\usepackage{accents}
\usepackage{afterpage}
\usepackage{amssymb}
\usepackage{appendix}
\usepackage{array}
\usepackage{bm}
\usepackage{bbm}
\usepackage{braket}
\usepackage{calc}
\usepackage{caption}
\usepackage{centernot}
\usepackage[noadjust]{cite}
\usepackage{dsfont}
\usepackage{enumitem}
\usepackage{float}
\restylefloat{figure}
\usepackage{graphicx}
\usepackage{indentfirst}
\usepackage{mathdots}
\usepackage{mathrsfs}
\usepackage{mathtools}
\usepackage{mleftright}
\usepackage{pgf}
\usepackage{scalerel}
\usepackage{stackengine}
\usepackage{tikz}
\usetikzlibrary{arrows}
\usepackage{tikz-cd}
\allowdisplaybreaks[4]

\newtheorem{thm}{Theorem}[section]

\newtheorem{cor}[thm]{Corollary}

\newtheorem{lemma}[thm]{Lemma}

\theoremstyle{definition}
\newtheorem{defn}[thm]{Definition}

\newtheorem{eg}[thm]{Example}
\newtheorem{notat}[thm]{Notation}

\theoremstyle{remark}
\newtheorem{rem}[thm]{Remark}

\DeclareMathOperator{\id}{id}

\DeclareMathOperator{\tr}{tr}

\DeclareMathOperator{\src}{src}
\DeclareMathOperator{\tar}{tar}
\DeclareMathOperator{\real}{Re}

\DeclareMathOperator{\mob}{M\ddot{o}b}

\newcommand{\N}{\mathbb{N}}

\newcommand{\R}{\mathbb{R}}
\newcommand{\C}{\mathbb{C}}

\newcommand{\D}{\mathbb{D}}
\newcommand{\pr}{\mathbb{P}}
\newcommand{\E}{\mathbb{E}}
\newcommand{\op}[1]{\operatorname{#1}}
\newcommand{\mcal}[1]{\mathcal{#1}}
\newcommand{\mfk}[1]{\mathfrak{#1}}

\newcommand{\mbf}[1]{\mathbf{#1}}
\newcommand{\indc}[1]{\mathds{1}\left\{#1\right\}}
\newcommand{\wtilde}[1]{\widetilde{#1}}

\newcommand{\deq}{\stackrel{d}{=}}

\DeclareRobustCommand{\SkipTocEntry}[5]{} 

\makeatletter
\newcommand\thefontsize[1]{{#1 The current font size is: \f@size pt\par}}
\makeatother
\begin{document}
\author{Benson Au}
\address{University of California, San Diego\\
         Department of Mathematics\\
         9500 Gilman Drive \# 0112\\
         La Jolla, CA 92093-0112\\
         USA}
       \email{\href{mailto:bau@ucsd.edu}{bau@ucsd.edu}}
\title[Wigner matrices with permuted entries]{Semicircular families of general covariance from Wigner matrices with permuted entries}
\thanks{}
\date{\today}

\begin{abstract}\label{sec:abstract}
Let $(\sigma_N^{(i)})_{i \in I}$ be a family of symmetric permutations of the entries of a Wigner matrix $\mbf{W}_N$. We characterize the limiting traffic distribution of the corresponding family of dependent Wigner matrices $(\mbf{W}_N^{\sigma_N^{(i)}})_{i \in I}$ in terms of the geometry of the permutations. We also consider the analogous problem for the limiting joint distribution of $(\mbf{W}_N^{\sigma_N^{(i)}})_{i \in I}$. In particular, we obtain a description in terms of semicircular families with general covariance structures.  As a special case, we derive necessary and sufficient conditions for traffic independence as well as sufficient conditions for free independence.
\end{abstract}

\maketitle
\tableofcontents

\section{Introduction}\label{sec:intro}

Free probability is a noncommutative (NC) probability theory built on Voiculescu's free independence. In the landmark work \cite{Voi91}, Voiculescu showed that free independence describes the large dimension limit behavior of independent random matrices in many generic situations. In recent years, the necessity of independence in this paradigm has been challenged. Mingo and Popa showed that freeness from the transpose occurs in unitarily invariant ensembles \cite{MP16}. C\'{e}bron, Dahlqvist, and Male extended these results to a larger class of dependent matrices using the traffic probability framework \cite{CDM16}. We now know freeness to be a generic phenomenon for random matrices of cactus-type \cite{AM20}.

In a different direction, one can simply view the transpose as a particular permutation $\sigma_N$ of the entries of a matrix $\mbf{A}_N$. Naturally, one can then ask when such a permutation produces asymptotic freeness for the pair $(\mbf{A}_N, \mbf{A}_N^{\sigma_N})$. This question was first considered by Popa, who found a sufficient condition for asymptotic freeness in the case of the GUE using the Wick calculus \cite{Pop18}. Similar results exist for other ensembles \cite{HP18,PH19,MP19,MPS20,MP20,MPS21}. In this article, we extend the analysis to general Wigner matrices and general NC covariance structures.

\subsection{Statement of results}\label{sec:results}

First, we specify the matrix model under consideration.

\begin{defn}[Wigner matrix]\label{defn:wigner}
For each $N \in \N$, let $\big(\mbf{X}_N(j, k)\big)_{1 \leq j \leq k \leq N}$ be a family of independent random variables such that:
\begin{enumerate}[label=(\roman*)]
\item the off-diagonal entries $\big(\mbf{X}_N(j, k)\big)_{1 \leq j < k \leq N}$ are complex-valued, centered, and of unit variance and pseudovariance $\E\big[\mbf{X}_N(j, k)^2\big] = \beta \in [-1, 1]$;
\item the diagonal entries $\big(\mbf{X}_N(j, j)\big)_{1 \leq j \leq N}$ are real-valued and of finite variance;
\item we have a strong uniform control on the moments
\begin{equation}\label{eq:moment_bound}
  \sup_{N \in \N} \sup_{1 \leq j \leq k \leq N} \E\big[|\mbf{X}_N(j, k)|^\ell\big] \leq m_\ell < \infty.
\end{equation}
\end{enumerate}
We call the random self-adjoint matrix defined by
\[
  \mbf{W}_N(j, k) = \frac{1}{\sqrt{N}}\mbf{X}_N(j, k)
\]
a \emph{Wigner matrix}.
\end{defn}

For a tuple $(j, k) \in [N]^2$, we define the transpose permutation
\[
  \intercal: [N]^2 \to [N]^2, \qquad (j, k) \mapsto (k, j).
\]
We say that a permutation $\sigma_N: [N]^2 \to [N]^2$ is \emph{symmetric} if it commutes with the transpose $\sigma_N \circ \intercal = \intercal \circ \sigma_N$, in which case the permuted matrix
\[
  \mbf{W}_N^{\sigma_N}(i, j) = \mbf{W}_N(\sigma_N(i, j))
\]
is again a Wigner matrix. Note that a symmetric permutation necessarily maps the diagonal to the diagonal. Moreover, if $\sigma_N, \rho_N \in \mfk{S}([N]^2)$, then
\[
  (\mbf{W}_N^{\sigma_N})^{\rho_N} = \mbf{W}_N^{\sigma_N \circ \rho_N}.
\]
Hereafter, when we refer to a matrix $\mbf{A}_N$ (resp., permutation $\sigma_N$), we implicitly refer to a sequence of matrices $(\mbf{A}_N)_{N \in \N}$ (resp., permutations $(\sigma_N)_{N \in \N}$).

Consider a family of symmetric permutations $(\sigma_N^{(i)})_{i \in I}$ and the corresponding family of (dependent) Wigner matrices $(\mbf{M}_N^{(i)})_{i \in I} = (\mbf{W}_N^{\sigma_N^{(i)}})_{i \in I}$. Popa showed that if $\mbf{W}_N \deq \op{GUE}(N, \frac{1}{N})$, then
\begin{equation}\label{eq:popa_condition}
  \#\big(\{(j, k, l) \in [N]^3 : \sigma_N^{(i)}(j, k) \in \{\sigma_N^{(i')}(j, l), \sigma_N^{(i')}(l, k)\}\}\big) = o_{i, i'}(N^2), \qquad \forall i \neq i' \in I
\end{equation}
is a sufficient condition for the asymptotic freeness of the family $(\mbf{M}_N^{(i)})_{i \in I}$ \cite[Theorem 3.2]{Pop18}. In other words, $(\mbf{M}_N^{(i)})_{i \in I}$ converges to a semicircular system $(s_i)_{i \in I}$. We extend this result in the following ways. First, we consider Wigner matrices in the generality of Definition \ref{defn:wigner}. Notably, condition \eqref{eq:popa_condition} is not sufficient for general $\beta$. Second, we consider general covariance structures for the limiting family $(s_i)_{i \in I}$. Finally, we determine necessary and sufficient conditions for asymptotic traffic independence (resp., sufficient conditions for asymptotic freeness).

To motivate the results, we consider the informative case of a pair of matrices $(\mbf{W}_N, \mbf{W}_N^{\sigma_N})$. Heuristically, the independence of the entries of $\mbf{W}_N$ suggests that if we can typically avoid repeating a permuted entry (i.e., avoid the situation of $\mbf{W}_N^{\sigma_N}(j, k) = \mbf{W}_N(j, k)$) when expanding the trace, then the calculation should follow the usual case of independent Wigner matrices. As a first guess, one might then expect that a vanishing proportion of fixed points 
\[
  \op{FP}(\sigma_N) = \{(j, k) \in [N^2] : \sigma_N(j, k) = (j ,k)\} = o(N^2)
\]
would suffice; however, the location of the permuted entry plays a crucial role. For example, in the case of the transpose $\sigma_N = \intercal$, the pair $(\mbf{W}_N, \mbf{W}_N^{\intercal})$ converges to a semicircular family of covariance $\begin{pmatrix} 1 & \beta \\ \beta & 1 \end{pmatrix}$ \cite[Corollary 3.9]{AM20}. This leads us to consider the set of transposed entries
\[
  \op{TP}(\sigma_N) = \{(j, k) \in [N^2] : \sigma_N(j, k) = (k ,j)\}.
\]
The case of $\beta = 1$ suggests that a nontrivial proportion of fixed points should contribute to the covariance accordingly. As a second guess, one might then expect that the limits
\begin{equation}\label{eq:inhomogeneous}
\begin{aligned}
  \lim_{N \to \infty} \frac{\#\big(\text{FP}(\sigma_N)\big)}{N^2} = a; \\
  \lim_{N \to \infty} \frac{\#\big(\text{TP}(\sigma_N)\big)}{N^2} = b,
\end{aligned}
\end{equation}
would produce a semicircular family of covariance $\begin{pmatrix} 1 & a+b\beta \\ a+b\beta & 1 \end{pmatrix}$; however, the multiplicative structure of a semicircular family requires a certain homogeneity in the proportion of fixed points (resp., transposed points) in any given row that the limits \eqref{eq:inhomogeneous} alone do not impose. For example, consider the permutation
\[
  \rho_N(j, k) = \begin{dcases}
    (j, k) &\text{if } \max(j, k) \in 2\N + 1; \\
    (k, j) &\text{if } \max(j, k) \in 2\N.
    \end{dcases}
\]
See Figure \ref{figure:inhomogeneous} for an illustration. A straightforward calculation shows that
\[
  \lim_{N \to \infty} \frac{\#\big(\text{FP}(\rho_N)\big)}{N^2} = \lim_{N \to \infty} \frac{\#\big(\text{TP}(\rho_N)\big)}{N^2} = \frac{1}{2},
\]
whereas
\[
  \lim_{N \to \infty} \E[\tr(\mbf{W}_N\mbf{W}_N^{\rho_N}\mbf{W}_N\mbf{W}_N^{\rho_N})] = \frac{2(\beta^2 + \beta + 1)}{3} \neq 2\left(\frac{\beta + 1}{2}\right)^2
\]
unless $\beta = 1$. This leads us to consider the refinement
\begin{align*}
  \op{FP}(\sigma_N, j) = \{k \in [N] : (j, k) \in \op{FP}(\sigma_N)\}; \\
  \op{TP}(\sigma_N, j) = \{k \in [N] : (j, k) \in \op{TP}(\sigma_N)\}.
\end{align*}

\begin{figure}
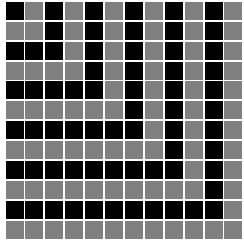
\caption{The action of $\rho_N$: black entries are fixed and gray entries are sent to their transpose.}\label{figure:inhomogeneous}
\end{figure}

But even this is not enough: as we will see, certain locations contribute to the trace expansion at a higher rate, leading to nonstandard joint distributions. For example, consider the operator-valued matrices
\[
  A_1 = \frac{1}{\sqrt{3}}\begin{pmatrix}
    s_1 & c_1 & c_2 \\
    c_1^* & s_2 & c_3 \\
    c_2^* & c_3^* & s_3
  \end{pmatrix}, \quad
  A_2 = \frac{1}{\sqrt{3}}\begin{pmatrix}
    s_3 & c_3 & c_1 \\
    c_3^* & s_1 & c_2 \\
    c_1^* & c_2^* & s_2
  \end{pmatrix},
\]
where $s_1, s_2, s_3, c_1, c_2, c_3 \in (\mcal{A}, \varphi)$ are $*$-free with $s_1, s_2, s_3$ standard semicircular and $c_1, c_2, c_3$ standard circular. Then $A_1, A_2 \in (\op{Mat}_{3 \times 3}(\C) \otimes \mcal{A}, \tr \otimes \varphi)$ are standard semicircular with
\begin{align*}
  (\tr \otimes \varphi)(A_1A_2) &= 0; \\
  (\tr \otimes \varphi)(A_1A_1A_2A_2) &= \frac{29}{27} \neq 1. 
\end{align*}
The additional contributions to the mixed fourth moment come from the terms
\begin{equation}\label{eq:overcount}
  \varphi(c_1c_1^*c_1c_1^*) = \varphi(c_2^*c_2c_2^*c_2) = 2 \neq 1.
\end{equation}
One can easily construct a sequence of permutations $(\sigma_N)_{N \in \N}$ with $\op{FP}(\sigma_N) = \op{TP}(\sigma_N) = \emptyset$ such that the pair $(\mbf{W}_N, \mbf{W}_N^{\sigma_N})$ converges to $(A_1, A_2)$ in distribution.

To ensure that our limit is distributed as a semicircular family in the traffic sense, we must avoid the possibility of the permuted entries meeting up in a way that overcontributes as in \eqref{eq:overcount}. As it turns out, this only happens if the entries are moved within a grid. 

\begin{defn}\label{defn:grid}
For an entry $(j, k) \in [N]^2$, we define the subset
\[
  \op{Grid}(j, k) = \{(l, m) \in [N]^2 : \{l, m\} \cap \{j, k\} \neq \emptyset\}\setminus\{(j, k), (k, j)\}.
\]
For a permutation $\sigma_N \in \mfk{S}([N]^2)$, we define 
\begin{align*}
  \op{Grid}(\sigma_N) = \{(j, k) \in [N]^2 : \sigma_N(j, k) \in \op{Grid}(j, k)\}.
\end{align*}
We say that $\sigma_N$ \emph{stays off the grid} if $\#(\op{Grid}(\sigma_N)) = o(N^2)$. See Figure \ref{fig:grid} for an illustration.
\end{defn}

\begin{figure}
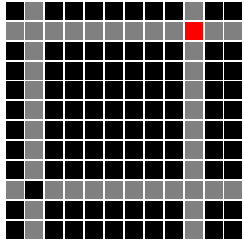
\caption{An example of an entry (in red) and its grid (in gray).}\label{fig:grid}
\end{figure}

Note that the grid condition is stable with respect to inversion: if $\sigma_N$ stays off the grid, then $\sigma_N^{-1}$ stays off the grid. For notational convenience, we write $(\mbf{M}_N^{(i)})_{i \in I} = (\mbf{W}_N^{\sigma_N^{(i)}})_{i \in I}$. We can now state our main results.

\begin{thm}\label{thm:permuted_traffic}
Let $(\sigma_N^{(i)})_{i \in I}$ be a family of symmetric permutations. First, suppose that $\beta \in (-1, 1)$. If
\begin{enumerate}[label=(\roman*)]
\item \label{cond:fp} $\displaystyle \lim_{N \to \infty} \max_{j \in [N]} \frac{\#\Big(\text{\emph{FP}}\big((\sigma_N^{(i')})^{-1} \circ \sigma_N^{(i)}, j\big)\Big)}{N} = \lim_{N \to \infty} \min_{j \in [N]} \frac{\#\Big(\text{\emph{FP}}\big((\sigma_N^{(i')})^{-1} \circ \sigma_N^{(i)}, j\big)\Big)}{N} = a_{i, i'}$;
\item \label{cond:tp} $\displaystyle \lim_{N \to \infty} \max_{j \in [N]} \frac{\#\Big(\text{\emph{TP}}\big((\sigma_N^{(i')})^{-1} \circ \sigma_N^{(i)}, j\big)\Big)}{N} = \lim_{N \to \infty} \min_{j \in [N]} \frac{\#\Big(\text{\emph{TP}}\big((\sigma_N^{(i')})^{-1} \circ \sigma_N^{(i)}, j\big)\Big)}{N} = b_{i, i'}$;
\item \label{cond:grid} $(\sigma_N^{(i')})^{-1} \circ \sigma_N^{(i)}$ stays off the grid for every $i, i' \in I$,
\end{enumerate}
then $(\mbf{M}_N^{(i)})_{i \in I}$ converges in traffic distribution to a semicircular traffic family $(t_i)_{i \in I}$ of covariance
\[
    \begin{tikzpicture}
    \node at (-1.25, 0) {$\mbf{K}(i, i') = \tau\bigg[$};
    \node at (2.5, 0) {$\bigg] = a_{i, i'} + b_{i, i'}\beta$};
    \draw[fill=black] (0,0) circle (1.75pt);
    \draw[fill=black] (1,0) circle (1.75pt);
    \draw[semithick, ->] (.125, .1) to node[midway, above] {\footnotesize$t_i$\normalsize} (.875, .1);
    \draw[semithick, ->] (.875, -.1) to node[midway, below] {\footnotesize$t_{i'}$\normalsize} (.125, -.1);
  \end{tikzpicture}  
\]
and pseudocovariance
\[
    \begin{tikzpicture}
    \node at (-1.25, 0) {$\mbf{J}(i, i') = \tau\bigg[$};
    \node at (2.5, 0) {$\bigg] = a_{i, i'}\beta  + b_{i, i'}$};
    \draw[fill=black] (0,0) circle (1.75pt);
    \draw[fill=black] (1,0) circle (1.75pt);
    \draw[semithick, <-] (.125, .1) to node[midway, above] {\footnotesize$t_i$\normalsize} (.875, .1);
    \draw[semithick, ->] (.875, -.1) to node[midway, below] {\footnotesize$t_{i'}$\normalsize} (.125, -.1);
  \end{tikzpicture}  
\]
Conversely, if $(\mbf{M}_N^{(i)})_{i \in I}$ converges in traffic distribution to $(t_i)_{i \in I}$ as above, then
\begin{enumerate}[label=(\Roman*)]
\item \label{Cond:fp} $\displaystyle \lim_{N \to \infty} \frac{\#\Big(\text{\emph{FP}}\big((\sigma_N^{(i')})^{-1} \circ \sigma_N^{(i)}\big)\Big)}{N^2} = a_{i, i'}$;
\item \label{Cond:tp} $\displaystyle \lim_{N \to \infty} \frac{\#\Big(\text{\emph{TP}}\big((\sigma_N^{(i')})^{-1} \circ \sigma_N^{(i)}\big)\Big)}{N^2} = b_{i, i'}$;
\item \label{Cond:grid} $(\sigma_N^{(i')})^{-1} \circ \sigma_N^{(i)}$ stays off the grid for every $i, i' \in I$.
\end{enumerate}
If $\beta \in \{-1, 1\}$, then the result holds with conditions \ref{cond:fp} and \ref{cond:tp} replaced by the weaker condition
\begin{align*}
  &\lim_{N \to \infty} \max_{j \in [N]} \frac{\#\Big(\text{\emph{FP}}\big((\sigma_N^{(i')})^{-1} \circ \sigma_N^{(i)}, j\big)\Big) + \#\Big(\text{\emph{TP}}\big((\sigma_N^{(i')})^{-1} \circ \sigma_N^{(i)}, j\big)\Big)\beta}{N} \\
  = &\lim_{N \to \infty} \min_{j \in [N]} \frac{\#\Big(\text{\emph{FP}}\big((\sigma_N^{(i')})^{-1} \circ \sigma_N^{(i)}, j\big)\Big) + \#\Big(\text{\emph{TP}}\big((\sigma_N^{(i')})^{-1} \circ \sigma_N^{(i)}, j\big)\Big)\beta}{N} = a_{i, i'} + b_{i, i'}\beta;
\end{align*}
conditions \ref{Cond:fp} and \ref{Cond:tp} replaced by the weaker condition
\begin{enumerate}[label=(\Roman*)]
\setcounter{enumi}{3}
\item \label{Cond:fptp} $\displaystyle \lim_{N \to \infty} \frac{\#\Big(\text{\emph{FP}}\big((\sigma_N^{(i')})^{-1} \circ \sigma_N^{(i)}\big)\Big) + \#\Big(\text{\emph{TP}}\big((\sigma_N^{(i')})^{-1} \circ \sigma_N^{(i)}\big)\Big) \beta}{N^2} = a_{i, i'} + b_{i, i'}\beta;$
\end{enumerate}
and, if $\beta = -1$, \ref{Cond:grid} replaced by the weaker condition\small
\[
  \lim_{N \to \infty} \frac{\#\big(\{(j, k, l) \in [N]^3: ((\sigma_N^{(i')})^{-1} \circ \sigma_N^{(i)})(j, k) = (l, k)\}\big) - \#\big(\{(j, k, l) \in [N]^3: ((\sigma_N^{(i')})^{-1} \circ \sigma_N^{(i)})(j, k) = (k, l)\}\big)}{N^2} = 0.
\]\normalsize

In particular, if $\beta \in (-1, 1)$ (resp., $\beta = 1$), then conditions \ref{Cond:fp} and \ref{Cond:tp} (resp., condition \ref{Cond:fptp}) with $a_{i, i'} = b_{i, i'} = 0$ for $i \neq i'$ and condition \ref{Cond:grid} are necessary and sufficient for the asymptotic traffic independence of the $(\mbf{M}_N^{(i)})_{i \in I}$.
\end{thm}

\begin{rem}\label{rem:exceptional_beta_case_traffic}
The reader may wonder why $\beta = -1$ is an exceptional case. Roughly speaking, the issue arises from precise cancellations between the variance and the pseudovariance. See Remarks \ref{rem:perfect_cancellation_exceptional_beta} and \ref{rem:mixed_double_trees_exceptional_beta} for more details.
\end{rem}

We would like to use the cactus-cumulant correspondence in \cite[\S 3.1]{AM20} to translate Theorem \ref{thm:permuted_traffic} to a more familiar free probabilistic statement; however, not all of the observables in the traffic framework appear in the usual NC framework. Restricting our analysis to this smaller class of observables, we obtain the following corollary.

\begin{cor}\label{cor:permuted_nc}
Let $(\sigma_N^{(i)})_{i \in I}$ be a family of symmetric permutations. Suppose that $\beta \in [-1, 1]$. If
\begin{enumerate}[label=(\alph*)]
\item \label{ncond:fp_tp} 
  \begin{align*}
  &\lim_{N \to \infty} \max_{j \in [N]} \frac{\#\Big(\text{\emph{FP}}\big((\sigma_N^{(i')})^{-1} \circ \sigma_N^{(i)}, j\big)\Big) + \#\Big(\text{\emph{TP}}\big((\sigma_N^{(i')})^{-1} \circ \sigma_N^{(i)}, j\big)\Big)\beta}{N} \\
  = &\lim_{N \to \infty} \min_{j \in [N]} \frac{\#\Big(\text{\emph{FP}}\big((\sigma_N^{(i')})^{-1} \circ \sigma_N^{(i)}, j\big)\Big) + \#\Big(\text{\emph{TP}}\big((\sigma_N^{(i')})^{-1} \circ \sigma_N^{(i)}, j\big)\Big)\beta}{N} = c_{i, i'};
  \end{align*}
\item \label{ncond:grid} $(\sigma_N^{(i')})^{-1} \circ \sigma_N^{(i)}$ stays off the grid for every $i, i' \in I$,
\end{enumerate}
then $(\mbf{M}_N^{(i)})_{i \in I}$ converges in distribution to a semicircular family $(s_i)_{i \in I}$ of covariance $\mbf{K}(i, i') = c_{i, i'}$. If $\beta = 0$, then we can relax the grid condition \ref{ncond:grid} to Popa's condition for the GUE \eqref{eq:popa_condition}. If one is only interested in freeness $\mbf{K}(i, i') = \indc{i = i'}$, then the homogeneity condition \ref{ncond:fp_tp} can be replaced by the averaged quantity
\begin{enumerate}[label=(\Alph*)]
\item \label{ncond:freeness} $\displaystyle \lim_{N \to \infty} \frac{\#\Big(\text{\emph{FP}}\big((\sigma_N^{(i')})^{-1} \circ \sigma_N^{(i)}\big)\Big) + \indc{\beta \neq 0}\#\Big(\text{\emph{TP}}\big((\sigma_N^{(i')})^{-1} \circ \sigma_N^{(i)}\big)\Big)}{N^2} = 0$.
\end{enumerate}
\end{cor}

\begin{rem}\label{rem:popa} Popa's sufficient condition for the GUE amounts to an asymmetric version of staying off the grid: in Figure \ref{fig:grid}, one simply needs to avoid the entries in the same row or the same column as the red entry. Indeed, the entries in the remainder of the grid do not factor into the calculations for the GUE since the pseudovariance $\beta = 0$ in this case. For general $\beta$, one needs to rule out the contributions from such arrangements. For example, consider the symmetric permutation defined by
  \[
    \eta_N(j, k) = \begin{dcases}
    (k, j+1) &\text{if } j < k - 1 \equiv 0 \bmod 2; \\
    (k, 1) &\text{if } j = k - 1 \equiv 0 \bmod 2; \\
    (k, j-1) &\text{if } 1 < j < k \equiv 0 \bmod 2; \\
    (k, k-1) &\text{if } 1 = j < k \equiv 0 \bmod 2; \\
    (j+1, k+1) &\text{if } j = k.
    \end{dcases}
  \]
See Figure \ref{fig:popa_failure} for an illustration. Then
  \[
    \op{FP}(\eta_N) = \op{TP}(\eta_N) = \{(j, k, l) \in [N]^3 : \eta_N(j, k) \in \{(j, l), (l, k)\}\} = \emptyset;
  \]
however, $\#(\op{Grid}(\eta_N)) \sim N^2$. A straightforward calculation shows that
  \begin{align*} 
    \lim_{N \to \infty} \E[\tr(\mbf{W}_N\mbf{W}_N^{\eta_N})] &= 0; \\
    \lim_{N \to \infty} \E[\tr(\mbf{W}_N\mbf{W}_N^{\eta_N}\mbf{W}_N\mbf{W}_N^{\eta_N})] &= \frac{\beta^2}{3},
  \end{align*}
which implies that $(\mbf{W}_N, \mbf{W}_N^{\eta_N})$ does not converge to a semicircular family for $\beta \neq 0$ (and so condition \eqref{eq:popa_condition} is no longer sufficient). This discrepancy is brought to the fore in the traffic framework, where the additional linear algebraic structure of the graph operations \emph{necessitates} the grid condition. 
\end{rem}

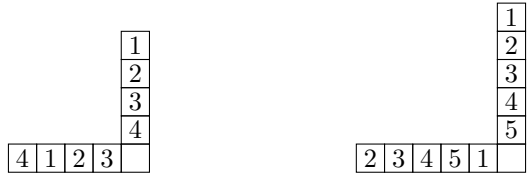
\begin{figure}
  \begin{tikzpicture}[baseline=(current  bounding  box.center)]
    \draw (-5,0) rectangle ++(0.375,0.375) node[pos=.5] {};
    \draw[draw=black] (-5,.375) rectangle ++(0.375,0.375) node[pos=.5] {4};
    \draw[draw=black] (-5,.75) rectangle ++(0.375,0.375) node[pos=.5] {3};
    \draw[draw=black] (-5,1.125) rectangle ++(0.375,0.375) node[pos=.5] {2};
    \draw[draw=black] (-5,1.5) rectangle ++(0.375,0.375) node[pos=.5] {1};

    \draw[draw=black] (-5.375,0) rectangle ++(0.375,0.375) node[pos=.5] {3};
    \draw[draw=black] (-5.75,0) rectangle ++(0.375,0.375) node[pos=.5] {2};
    \draw[draw=black] (-6.125,0) rectangle ++(0.375,0.375) node[pos=.5] {1};
    \draw[draw=black] (-6.5,0) rectangle ++(0.375,0.375) node[pos=.5] {4};

    \draw (0,0) rectangle ++(0.375,0.375) node[pos=.5] {};
    \draw[draw=black] (0,.375) rectangle ++(0.375,0.375) node[pos=.5] {5};
    \draw[draw=black] (0,.75) rectangle ++(0.375,0.375) node[pos=.5] {4};
    \draw[draw=black] (0,1.125) rectangle ++(0.375,0.375) node[pos=.5] {3};
    \draw[draw=black] (0,1.5) rectangle ++(0.375,0.375) node[pos=.5] {2};
    \draw[draw=black] (0,1.875) rectangle ++(0.375,0.375) node[pos=.5] {1};

    \draw[draw=black] (-.375,0) rectangle ++(0.375,0.375) node[pos=.5] {1};
    \draw[draw=black] (-.75,0) rectangle ++(0.375,0.375) node[pos=.5] {5};
    \draw[draw=black] (-1.125,0) rectangle ++(0.375,0.375) node[pos=.5] {4};
    \draw[draw=black] (-1.5,0) rectangle ++(0.375,0.375) node[pos=.5] {3};
    \draw[draw=black] (-1.875,0) rectangle ++(0.375,0.375) node[pos=.5] {2};
\end{tikzpicture}
\caption{The action of $\eta_N$ on off-diagonal entries: on the left, $k \equiv 1 \bmod 2$; on the right, $k \equiv 0 \bmod 2$. In both cases, entries with the same number swap positions.}\label{fig:popa_failure}
\end{figure}

\begin{rem}\label{rem:general_beta}
If $b_{i, i'} \neq 0$, then the restriction to $\beta \in [-1, 1]$ is necessary to ensure convergence to a semicircular family in certain cases. For example, if $\beta \in \D$, then a straightforward calculation shows that
\begin{align*}
  \lim_{N \to \infty} \E[\tr(\mbf{W}_N\mbf{W}_N^\intercal)] &= \real(\beta); \\
  \lim_{N \to \infty} \E[\tr(\mbf{W}_N\mbf{W}_N\mbf{W}_N^\intercal\mbf{W}_N^\intercal)] &= 1 + \frac{2}{3}|\beta|^2 + \frac{1}{6}(\beta^2 + \overline{\beta}^2).
\end{align*}
If $(\mbf{W}_N, \mbf{W}_N^\intercal)$ converges to a semicircular family, then the calculations above would imply that
\[
  1 + \real(\beta)^2 = 1 + \frac{2}{3}|\beta|^2 + \frac{1}{6}(\beta^2 + \overline{\beta}^2),
\]
which reduces to
\[
  (\beta - \overline{\beta})^2 = 0,
\]
and so $\beta \in \D \cap \R = [-1, 1]$.

The obstruction from the transposed entries in the case of strictly complex $\beta$ becomes apparent in the traffic framework. Even in the usual case of independent Wigner matrices, the limiting traffic distribution for such $\beta$ is unwieldy \cite[Equation (3.16)]{Au18a}. One can obtain a similarly unwieldy formula for the limiting traffic distribution of $(\mbf{M}_N^{(i)})_{i \in I}$. On the other hand, the transposed entries are the sole obstruction to convergence to a semicircular family. In particular, Corollary \ref{cor:permuted_nc} still holds for general $\beta \in \D$ with the $\#\Big(\text{TP}\big((\sigma_N^{(i')})^{-1} \circ \sigma_N^{(i)}, j\big)\Big)\beta$ terms removed from \ref{ncond:fp_tp} under the additional assumption that
\begin{enumerate}[label=(\alph*)]
\setcounter{enumi}{2}
\item \label{ncond:general_beta} $\displaystyle \lim_{N \to \infty} \frac{\#\Big(\text{TP}\big((\sigma_N^{(i')})^{-1} \circ \sigma_N^{(i)}\big)\Big)}{N^2} = 0$.
\end{enumerate}
\end{rem}

\subsection{Examples and simulations}\label{sec:examples_and_simulations}

Let $(\mbf{A}_N, \mbf{B}_N)$ be a pair of Wigner matrices. If $\mbf{A}_N$ and $\mbf{B}_N$ were independent, then they would be asymptotically free and the empirical spectral distribution of the anticommutator $[\mbf{A}_N, \mbf{B}_N] = \mbf{A}_N\mbf{B}_N + \mbf{B}_N\mbf{A}_N$ would converge to the symmetric Poisson distribution $\nu_{\mcal{SP}}$ as $N \to \infty$ \cite{NS98}, where
\[
  \nu_{\mcal{SP}}(dx) = \frac{\sqrt{3}\indc{|x| \leq \sqrt{\frac{11 + 5\sqrt{5}}{2}}}}{2\pi|x|}\left(\frac{3x^2 + 1}{9\sqrt[3]{\frac{18x^2+1}{27} - \sqrt{\frac{x^2(1+11x^2-x^4)}{27}}}} - \sqrt[3]{\frac{18x^2+1}{27} - \sqrt{\frac{x^2(1+11x^2-x^4)}{27}}}\right).
\]
For our simulations, we will instead consider dependent Wigner matrices $(\mbf{A}_N, \mbf{B}_N) = (\mbf{W}_N, \mbf{W}_N^{\sigma_N})$. The values of $\beta$ and the permutations $\sigma_N$ will be chosen to satisfy Corollary \ref{cor:permuted_nc} and produce asymptotic freeness. To demonstrate this, we plot the eigenvalues of the anticommutator $[\mbf{A}_N, \mbf{B}_N]$ against the limiting density $\nu_{\mcal{SP}}(dx)$.

\begin{eg}\label{eg:sporadic_transpose}
For $n \in \N$, let
  \[
    \zeta_{N, n}(j, k) = \begin{dcases}
    (j, k) &\text{if } j + k \equiv 0 \bmod n; \\
    (k, j) &\text{if } j + k \not\equiv 0 \bmod n. 
    \end{dcases}
  \]
Of course, if $n = 1$, then $\zeta_{N, n} = \id_N$. In general, $\zeta_{N, n}$ fixes approximately $\frac{1}{n}$ of the entries and sends $\frac{n-1}{n}$ of the entries to their transpose in a way that satisfies the homogeneity condition \ref{ncond:fp_tp}. Note that negative values of $\beta$ allow one to in some sense eliminate dependence: we illustrate this by choosing $n$ and $\beta$ such that $\frac{1}{n} + \frac{n-1}{n}\beta = 0$. See Figure \ref{figure:cancellation}.
\end{eg}

\begin{eg}\label{eg:non_gaussian}
Let $\sigma_N$ be the transpose with respect to the anti-diagonal. In other words, 
  \[
    \sigma_N(j, k) = (N + 1 - k, N + 1 - j).
  \]
  One easily sees that
  \[
    \#\big(\text{FP}(\sigma_N) \cup \text{TP}(\sigma_N)\big) + \#\big(\text{Grid}(\sigma_N)\big) = o(N^2).
  \] 
Corollary \ref{cor:permuted_nc} and Remark \ref{rem:general_beta} then guarantee asymptotic freeness for the pair $(\mbf{W}_N, \mbf{W}_N^{\sigma_N})$ for any $\beta \in \D$. We choose non-Gaussian distributions for the entries of $\mbf{W}_N$ to further emphasize this universality. See Figure \ref{figure:anti_transpose}.
\end{eg}

\begin{figure}
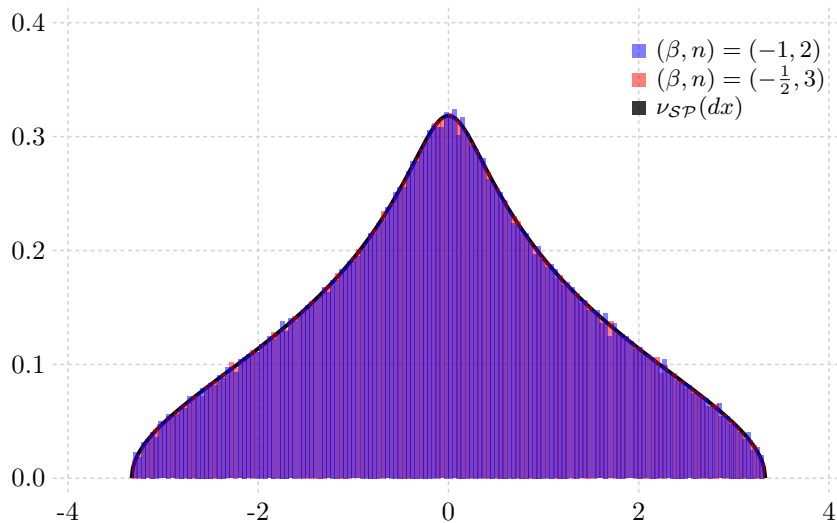
\caption{Histograms of the empirical spectral distribution of the anticommutator $[\mbf{W}_N, \mbf{W}_N^{\zeta_{N, n}}]$ overlaid for two values of $(\beta, n)$ and plotted against the limiting density $\nu_{\mcal{SP}}(dx)$. The overlapping region takes on the color blue + red = purple. The Wigner matrix $\mbf{W}_N$ is chosen to be Gaussian with $N = 10000$.}\label{figure:cancellation}
\end{figure}

\begin{figure}
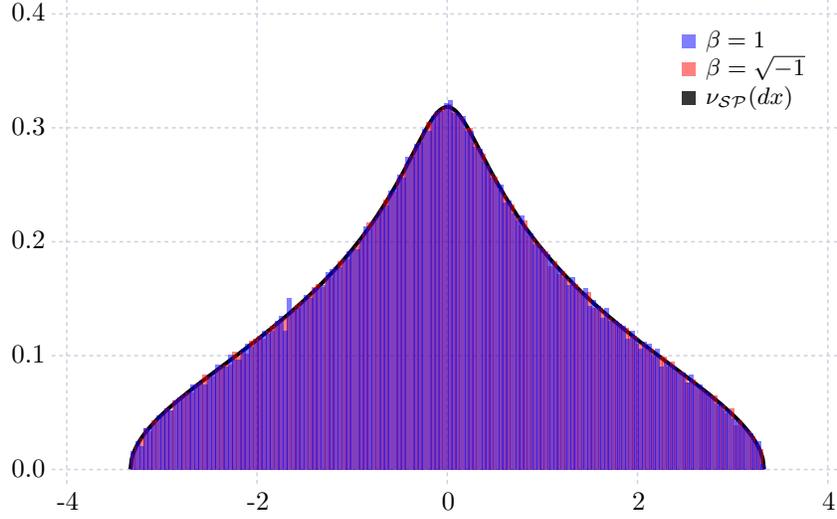
\caption{Histograms of the empirical spectral distribution of the anticommutator $[\mbf{W}_N, \mbf{W}_N^{\sigma_N}]$ overlaid for two values of $\beta$ and plotted against the limiting density $\nu_{\mcal{SP}}(dx)$. The overlapping region takes on the color blue + red = purple. The blue corresponds to a real symmetric Rademacher Wigner matrix ($\beta = 1$); the red corresponds to a complex Hermitian analogue of the Rademacher where the real and imaginary parts of the upper triangular entries are identically equal $X+iX$ ($\beta = \sqrt{-1}$); in both cases, $N = 10000$.}\label{figure:anti_transpose}
\end{figure}

\subsection{Background}\label{sec:background}
We briefly review the relevant aspects of the NC framework. We refer the reader to \cite{VDN92,NS06,MS17} for detailed treatments.

\begin{defn}[Free probability]\label{defn:free_probability}
A \emph{NC probability space (ncps)} is a pair $(\mcal{A}, \varphi)$ with $\mcal{A}$ a unital algebra over $\C$ and $\varphi: \mcal{A} \to \C$ a unital linear functional. The \emph{distribution} of a family of random variables $\mbf{a} = (a_i)_{i \in I} \subset \mcal{A}$ is the linear functional
\[
  \mu_{\mbf{a}}: \C\langle\mbf{x}\rangle \to \C, \qquad P \mapsto \varphi(P(\mbf{a})),
\]
where $\C\langle\mbf{x}\rangle$ is the free algebra over the indeterminates $\mbf{x} = (x_i)_{i \in I}$ and $P(\mbf{a}) \in \mcal{A}$ is the usual evaluation. A sequence of families $\mbf{a}_N = (a_N^{(i)})_{i \in I} \subset (\mcal{A}_N, \varphi_N)$ \emph{converges in distribution} if the corresponding sequence of distributions $(\mu_{\mbf{a}_N})_{N \in \N}$ converges pointwise. Note that the limit defines a new ncps $(\C\langle\mbf{x}\rangle, \lim_{N \to \infty} \mu_{\mbf{a}_N})$.

Let $(\mcal{NC}(n), \leq)$ denote the poset of noncrossing partitions of $[n]$ with the reversed refinement order. We write $0_n$ for the minimal element consisting of singletons; $1_n$ for the maximal element consisting of a single block; and $\mob_{\mcal{NC}}$ for the corresponding M\"{o}bius function. For a noncrossing partition $\pi \in \mcal{NC}(n)$, we define the multilinear functional
\[
  \varphi_\pi : \mcal{A}^n \to \C, \qquad (a_1, \ldots, a_n) \mapsto \prod_{B \in \pi} \varphi(B)[a_1, \ldots, a_n],
\] 
where a block $B = (i_1 < \cdots < i_m) \in \pi$ defines a partial product
\[
\varphi(B)[a_1, \ldots, a_n] = \varphi(a_{i_1} \cdots a_{i_m}).
\]
The \emph{free cumulant} $\kappa_\pi$ is the multilinear functional $\kappa_\pi : \mcal{A}^n \to \C$ obtained from the M\"{o}bius convolution
\begin{equation}\label{eq:free_cumulants_mobius}
  \kappa_\pi[a_1, \ldots, a_n] = \sum_{\substack{\sigma \in \mcal{NC}(n) \\ \text{s.t. } \sigma \leq \pi}} \varphi_\sigma[a_1, \ldots, a_n] \mob_{\mcal{NC}}(\sigma, \pi).
\end{equation}
By the M\"{o}bius inversion formula, this is equivalent to
\begin{equation}\label{eq:free_cumulants_sum}
  \varphi_\pi[a_1, \ldots, a_n] = \sum_{\substack{\sigma \in \mcal{NC}(n) \\ \text{s.t. } \sigma \leq \pi}} \kappa_\sigma[a_1, \ldots, a_n].
\end{equation}
We use the notation $\kappa_n = \kappa_{1_n}$, which allows us to formulate the multiplicative property of the free cumulants:
\[
  \kappa_\pi[a_1, \ldots, a_n] = \prod_{B \in \pi} \kappa(B)[a_1, \ldots, a_n],
\] 
where $B = (i_1 < \cdots < i_m)$ is a block as before and
\[
  \kappa(B)[a_1, \ldots, a_n] = \kappa_m[a_{i_1}, \ldots, a_{i_m}].
\]
Subsets $(\mcal{S}_i)_{i \in I}$ of a ncps $(\mcal{A}, \varphi)$ are \emph{freely independent} if for any $n \geq 2$ and $a_1, \ldots, a_n$ such that $a_j \in \mcal{S}_{i(j)}$,
\[
\exists i(j) \neq i(k) \implies \kappa_n[a_1, \ldots, a_n] = 0.
\]
A sequence of families $(\mbf{a}_N)_{N \in \N}$ is \emph{asymptotically free} if it converges in distribution to a limit $\mbf{x}$ that is free in $(\C\langle\mbf{x}\rangle, \lim_{N \to \infty} \mu_{\mbf{a}_N})$.
\end{defn}

We focus on one specific distribution in the NC framework. In particular, we recall the NC analogue of the multivariate normal distribution.

\begin{eg}[Semicircular family]\label{eg:semicircular_family}
Let $(\mbf{K}(i, i'))_{i, i' \in I}$ be a real positive semidefinite matrix. A family of random variables $\mbf{s} = (s_i)_{i \in I}$ in a ncps $(\mcal{A}, \varphi)$ is a \emph{(centered) semicircular family of covariance $\mbf{K}$} if
\[
  \kappa_2(s_i, s_{i'}) = \mbf{K}(i, i'), \qquad \forall i, i' \in I,
\]
and all other cumulants on $\mbf{s}$ vanish. Equivalently, for any $n \in \N$,
\begin{equation}\label{eq:semicircular_family}
  \varphi(s_{i(1)} \cdots s_{i(n)}) = \sum_{\pi \in \mcal{NC}_2(n)} \prod_{\{j, k\} \in \pi} \mbf{K}(i(j), i(k)),
\end{equation}
where $\mcal{NC}_2(n) \subset \mcal{NC}(n)$ is the subset of noncrossing pair partitions of $[n]$. In particular, the $(s_i)_{i \in I}$ are free iff $(\mbf{K}(i, i'))_{i, i' \in I}$ is diagonal, in which case we call $(s_i)_{i \in I}$ a \emph{semicircular system}.

A well-known result of Dykema proves that independent Wigner matrices $(\mbf{W}_N^{(i)})_{i \in I} \subset (L^{\infty-}(\Omega, \mcal{F}, \pr) \otimes \op{Mat}_N(\C), \E[\tr(\cdot)])$ converge in distribution to a semicircular system \cite[Theorem 2.1]{Dyk93}.
\end{eg}

We will also need some basics from the traffic framework. We refer the reader to \cite{Mal20,CDM16,AM20} for detailed treatments.

\begin{defn}[Traffic probability]\label{defn:traffic_probability}
A \emph{multidigraph} $G = (V, E, \src, \tar)$ consists of a nonempty set of vertices $V$, a set of edges $E$, and a pair of maps $\src, \tar: E \to V$ indicating the source and target of each edge. When the context is clear, we omit the maps $\src, \tar$ from the notation and simply write $G = (V, E)$.

Let $I$ be an index set. A \emph{test graph in $I$} is a pair $T = (G, \gamma)$ consisting of a finite connected multidigraph $G$ with edge labels $\gamma: E \to I$. For a partition $\pi \in \mcal{P}(V)$, we define $T^\pi$ as the quotient test graph obtained from $T$ by identifying vertices in the same block in $\pi$. Formally, $T^\pi = (V^\pi, E^\pi, \src^\pi, \tar^\pi, \gamma^\pi)$, where
\begin{enumerate}[label=(\roman*)]
\item $V^\pi = V/\sim_\pi = \{B : B \in \pi\}$;
\item $E^\pi = E$;
\item $\src^\pi = [\cdot]_\pi \circ \src$ and $\tar^\pi = [\cdot]_\pi \circ \tar$;
\item $\gamma^\pi = \gamma$.
\end{enumerate}
We write $\mcal{T}\langle I \rangle$ for the set of all test graphs in $I$ and $\C\mcal{T}\langle I \rangle$ for the complex vector space spanned by $\mcal{T}\langle I \rangle$.

For the purposes of this article, a \emph{traffic space} is a pair $(\mcal{A}, \tau)$ with $\mcal{A}$ a unital algebra over $\C$ and $\tau: \C\mcal{T}\langle \mcal{A} \rangle \to \C$ a unital linear functional called the \emph{traffic state}. Unital in this context simply means that $\tau[T_0] = 1$ for the trivial test graph $T_0$ consisting of a single isolated vertex. We also work with a transform of the traffic state called the \emph{injective traffic state}:
\[
  \tau^0: \C\mcal{T}\langle \mcal{A} \rangle \to \C, \qquad \tau^0[T] = \sum_{\pi \in \mcal{P}(V)} \tau[T^\pi]\mob_{\mcal{P}}(0_V, \pi),
\]
where $\mob_{\mcal{P}}$ is the M\"{o}bius function on the the poset of partitions $(\mcal{P}(V), \leq)$ with the reversed refinement order and $0_V$ is the singleton partition. By the M\"{o}bius inversion formula, this is equivalent to
\[
  \tau[T] = \sum_{\pi \in \mcal{P}(V)} \tau^0[T^\pi].
\]
Every traffic space $(\mcal{A}, \tau)$ defines a ncps $(\mcal{A}, \varphi_\tau)$, where
\begin{equation}\label{eq:traffic_to_ncps}
  \begin{tikzpicture}[baseline=(current  bounding  box.center), shorten > = 1.5pt]
    \node at (-2.5, 0) {$\varphi_\tau(a_1 \cdots a_n) = \tau\Bigg[$};
    \node at (1.5, 0) {$\Bigg]$.};
    \draw[fill=black] (.6, 0) circle (1pt);
    \draw[fill=black] (-.6, 0) circle (1pt);
    \draw[fill=black] (.3, .5196) circle (1pt);
    \draw[fill=black] (.3, -.5196) circle (1pt);
    \draw[fill=black] (-.3, .5196) circle (1pt);
    \draw[fill=black] (-.3, -.5196) circle (1pt);
    \draw[semithick, ->] (.6,0) to node[pos=.625, right] {${\scriptstyle a_{n-1}}$} (.3,-.5196);
    \draw[semithick, ->] (.3,-.5196) to node[midway, below] {${\scriptstyle \cdots}$} (-.3,-.5196);
    \draw[semithick, ->] (-.3,-.5196) to node[pos=.375, left] {${\scriptstyle a_3}$} (-.6, 0);
    \draw[semithick, ->] (-.6, 0) to node[pos=.625, left] {${\scriptstyle a_2}$} (-.3, .5196);
    \draw[semithick, ->] (-.3, .5196) to node[midway, above] {${\scriptstyle a_1}$} (.3, .5196);
    \draw[semithick, ->] (.3, .5196) to node[pos=.375, right] {${\scriptstyle a_n}$} (.6, 0);
  \end{tikzpicture}
\end{equation}

The \emph{traffic distribution} of a family of random variables $\mbf{a} = (a_i)_{i \in I} \subset \mcal{A}$ is the linear functional
\[
  \nu_{\mbf{a}}: \C\mcal{T}\langle I \rangle \to \C, \qquad T \mapsto \tau[T(\mbf{a})],
\]
where $T(\mbf{a}) \in \mcal{T}\langle\mcal{A}\rangle$ is the test graph obtained by replacing the edge labels $\gamma: E \to I$ with edge labels
\[
  \gamma_{\mbf{a}}: E \to \mcal{A}, \qquad e \mapsto a_{\gamma(e)}.
\]
In particular, the traffic distribution $\nu_{\mbf{a}}$ contains the information of the distribution $\mu_{\mbf{a}}$ via \eqref{eq:traffic_to_ncps}. A sequence of families $\mbf{a}_N = (a_N^{(i)})_{i \in I}$ \emph{converges in traffic distribution} if the corresponding sequence of traffic distributions $(\nu_{\mbf{a}_N})_{N \in \N}$ converges pointwise. Note that the limit defines a new traffic space $(\C\langle\mbf{x}\rangle, \lim_{N \to \infty} \nu_{\mbf{a}_N})$. The \emph{injective traffic distribution} and \emph{convergence in injective traffic distribution} are defined in the obvious way. Naturally, the M\"{o}bius relation implies that convergence in traffic distribution is equivalent to convergence in injective traffic distribution.

Let $T = (V, E, \gamma) \in \mcal{T}\langle J \rangle$ be a test graph in the disjoint union $J = \sqcup_{i \in I} J_i$. We think of each of the subsets $J_i$ as a different color for the edges via the labeling $\gamma: E \to J$. A \emph{colored component} in $T$ is a connected component in the subgraph obtained from $T$ by only keeping the edges coming from a single color $\gamma^{-1}(J_i) \subset E$. We write $\mcal{CC}_i(T)$ for the set of $i$-colored components in $T$ and $\mcal{CC}(T) = \sqcup_{i \in I} \mcal{CC}_i(T)$ for the set of all colored components. The \emph{graph of colored components} $\mcal{GCC}(T)$ is the simple bipartite graph $(V \sqcup \mcal{CC}(T), E_{\mcal{GCC}(T)})$ where adjacency is determined by inclusion in $T$: if $v \in V$ and $C \in \mcal{CC}(T)$, then
\[
  v \sim_{\mcal{GCC}(T)} C \iff v \text{ is a vertex of } C.
\]
Random variables $\mbf{a} = \sqcup_{i \in I} \mbf{a}_i$ with $\mbf{a}_{i} = (a_{i, j})_{j \in J_i}$ are \emph{traffic independent} if
\[
  \nu_{\mbf{a}}^0[T] = \indc{\mcal{GCC}(T) \text{ is a tree}}\prod_{C \in \mcal{CC}(T)} \nu_{\mbf{a}}^0[C], \qquad \forall T \in \mcal{T}\langle J\rangle. 
\]
A sequence of families $(\mbf{a}_N)_{N \in \N}$ is \emph{asymptotically traffic independent} if it converges in traffic distribution to a limit $\mbf{x}$ that is traffic independent in $(\C\langle\mbf{x}\rangle, \lim_{N \to \infty} \nu_{\mbf{a}_N})$.
\end{defn}

Similarly, we focus on one specific distribution in the traffic framework. It will be helpful to introduce some notation beforehand.

\begin{notat}\label{notat:graphs}
Let $G = (V, E, \src, \tar)$ be a multidigraph. We distinguish between loops $\mcal{L} \subset E$ and non-loop edges $\mcal{N} = \mcal{L}^c$. We define $\wtilde{G} = (V, \wtilde{E})$ as the undirected graph obtained from $G$ by omitting the direction and multiplicity of the edges. Formally, $\wtilde{E} = \{[e] : e \in E\}$, where $[e] = \{e' \in E : e \sim e'\}$ and
\[
  e \sim e' \iff \{\src(e), \tar(e)\} = \{\src(e'), \tar(e')\}.
\]
The partition $E = \mcal{L} \sqcup \mcal{N}$ then drops to a partition of the quotient $\wtilde{E}= \wtilde{\mcal{L}} \sqcup \wtilde{\mcal{N}}$. In particular, we can write the underlying simple graph in $G$ as $\underline{G} = (V, \wtilde{\mcal{N}})$. If our graph comes with edge labels $\gamma: E \to I$, then we record the undirected multiplicity of a label $i$ in a class of edges $[e]$ with
\[
  m_{[e], i} = \#\big(\gamma^{-1}(\{i\}) \cap [e]\big).
\]
We use the notation $m_{[e]} = \#([e]) = \sum_{i \in I} m_{[e], i}$ for the total multiplicity of the edge $[e]$. See Figure \ref{fig:graphs} for an illustration.
\end{notat}

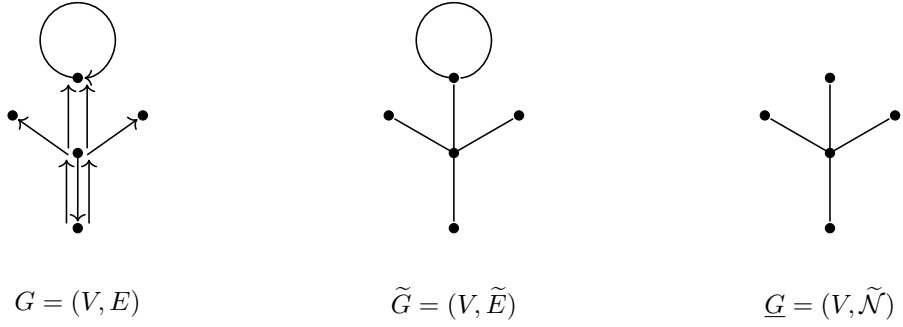
\begin{figure}
  \begin{tikzpicture}[baseline=(current  bounding  box.center), shorten > = 2.5pt]
    \node at (0, -1) {$G = (V, E)$};
    \draw[fill=black] (0, 0) circle (1.75pt);
    \draw[fill=black] (0, 1) circle (1.75pt);
    \draw[fill=black] (0, 2) circle (1.75pt);
    \draw[fill=black] (.866, 1.5) circle (1.75pt);
    \draw[fill=black] (-.866, 1.5) circle (1.75pt);
    \draw[semithick, ->] (.15,.0625) to (.15,1);
    \draw[semithick, ->] (0, 1) to (0,0);
    \draw[semithick, ->] (-.15,.0625) to (-.15, 1);
    
    \draw[semithick, ->] (.125,1.0625) to (.125, 2);
    \draw[semithick, ->] (-.125,1.0625) to (-.125, 2);

    \draw[semithick, ->] (.125,.975) to (.866, 1.5);

    \draw[semithick, ->] (-.125,.975) to (-.866, 1.5);

    \draw[semithick, ->] (0,2) arc (270:-90:.5cm);

    \node at (5, -1) {$\wtilde{G} = (V, \wtilde{E})$};
    \draw[fill=black] (5, 0) circle (1.75pt);
    \draw[fill=black] (5, 1) circle (1.75pt);
    \draw[fill=black] (5, 2) circle (1.75pt);
    \draw[fill=black] (5.866, 1.5) circle (1.75pt);
    \draw[fill=black] (4.134, 1.5) circle (1.75pt);
    \draw[semithick, -] (5, 1) to (5,0);
    
    \draw[semithick, -] (5,1.0625) to (5, 2);

    \draw[semithick, -] (5,1) to (5.866, 1.5);

    \draw[semithick, -] (5,1) to (4.134, 1.5);

    \draw[semithick, -] (5,2) arc (270:-90:.5cm);

    \node at (10, -1) {$\underline{G} = (V, \wtilde{\mcal{N}})$};
    \draw[fill=black] (10, 0) circle (1.75pt);
    \draw[fill=black] (10, 1) circle (1.75pt);
    \draw[fill=black] (10, 2) circle (1.75pt);
    \draw[fill=black] (10.866, 1.5) circle (1.75pt);
    \draw[fill=black] (9.134, 1.5) circle (1.75pt);
    \draw[semithick, -] (10, 1) to (10,0);
    
    \draw[semithick, -] (10,1.0625) to (10, 2);

    \draw[semithick, -] (10,1) to (10.866, 1.5);

    \draw[semithick, -] (10,1) to (9.134, 1.5);
  \end{tikzpicture}
\caption{An example of a multidigraph $G = (V, E)$, the undirected and multiplicity free graph $\wtilde{G} = (V, \wtilde{E})$, and the underlying simple graph $\underline{G} = (V, \wtilde{\mcal{N}})$.}\label{fig:graphs}
\end{figure}

\begin{eg}[Semicircular traffic family]\label{eg:semicircular_traffic_family}
A \emph{double tree} is a multidigraph $G = (V, E, \src, \tar)$ with no loops $\mcal{L} = \emptyset$ such that the underlying simple graph $\underline{G} = (V, \wtilde{\mcal{N}})$ is a tree with edge multiplicity $m_{[e]} = 2$ for every $[e] \in \wtilde{\mcal{N}}$. We call the edges $[e] = \{e, e'\} \in \wtilde{\mcal{N}}$ in a double tree \emph{twin edges}. Twin edges $\{e, e'\}$ are \emph{congruent} (resp., \emph{opposing}) if $\src(e) = \src(e')$ and $\tar(e) = \tar(e')$ (resp., $\src(e) = \tar(e')$ and $\tar(e) = \src(e')$).

Let $(\mbf{K}(i, i'))_{i, i' \in I}$ be a real positive semidefinite matrix and $(\mbf{J}(i, i'))_{i, i' \in I}$ a real symmetric matrix. A family of random variables $\mbf{s} = (s_i)_{i \in I}$ in a traffic space $(\mcal{A}, \tau)$ is a \emph{semicircular traffic family of covariance $\mbf{K}$ and pseudocovariance $\mbf{J}$} if
\begin{equation}\label{eq:semicircular_traffic_family}
\begin{aligned}
  \nu_{\mbf{s}}^0[T] = \indc{T \text{ is a double tree}}\prod_{[e] = \{e, e'\} \in \wtilde{\mcal{N}}} \Big(&\indc{\{e, e'\} \text{ are opposing}}\mbf{K}(\gamma(e), \gamma(e')) \\
  + &\indc{\{e, e'\} \text{ are congruent}}\mbf{J}(\gamma(e), \gamma(e'))\Big).
\end{aligned}
\end{equation}
In particular, the $(s_i)_{i \in I}$ are traffic independent iff $(\mbf{K}(i, i'))_{i, i' \in I}$ and $(\mbf{J}(i, i'))_{i, i' \in I}$ are diagonal.

A semicircular traffic family can be realized as follows. We define a traffic state on the algebra of random matrices $\mcal{A} = L^{\infty-}(\Omega, \mcal{F}, \pr) \otimes \op{Mat}_N(\C)$ by
\[
  \tau_N[T] = \frac{1}{N} \sum_{\phi: V \to [N]} \E\Big[\prod_{e \in E} \gamma(e)(\phi(\tar(e)), \phi(\src(e)))\Big], \qquad \forall T \in \mcal{T}\langle\mcal{A}\rangle.
\]
In particular, note that $\varphi_{\tau_N}(\cdot) = \E[\tr(\cdot)]$. For notational convenience, we abbreviate
\[
  (\phi(e)) := (\phi(\tar(e)), \phi(\src(e))).
\]
The injective traffic state admits a particularly simple form in this case: 
\begin{equation}\label{eq:injective_traffic_state_matrix}
  \tau_N^0[T] = \frac{1}{N} \sum_{\phi: V \hookrightarrow [N]} \E\Big[\prod_{e \in E} \gamma(e)(\phi(e))\Big], \qquad \forall T \in \mcal{T}\langle\mcal{A}\rangle,
\end{equation}
where $\phi: V \hookrightarrow [N]$ denotes an injective map. A result of Male proves that independent Wigner matrices $(\mbf{W}_N^{(i)})_{i \in I}$ converge in traffic distribution to a semicircular traffic family of covariance $\mbf{K}(i, i') = \indc{i = i'}$ and pseudocovariance $\mbf{J}(i, i') = \beta_i \indc{i = i'}$ \cite[Proposition 3.1]{Mal20}. See \cite{Au18a,Au20} for additional work on Wigner matrices and their generalizations in the traffic framework.
\end{eg}

As the name suggests, a semicircular traffic family $\mbf{a} = (a_i)_{i \in I} \subset (\mcal{A}, \tau)$ of covariance $\mbf{K}$ and pseudocovariance $\mbf{J}$ defines a semicircular family of covariance $\mbf{K}$ in $(\mcal{A}, \varphi_\tau)$. This follows from the cactus-cumulant correspondence \cite[\S 3.1]{AM20}; however, the converse does not hold. For example, Dykema's result for independent Wigner matrices $(\mbf{W}_N^{(i)})_{i \in I}$ holds for arbitrary pseudovariances $(\beta_i)_{i \in I} \subset \D$, whereas the case of strictly complex $\beta_i$ does not have the multiplicative structure of \eqref{eq:semicircular_traffic_family} if there are any congruent twin edges \cite[Equation (3.16)]{Au18a}. In particular, note that the distribution $\mu_{\mbf{a}}$ does not require the full knowledge of the traffic distribution $\nu_{\mbf{a}}$. Indeed, the definition \eqref{eq:traffic_to_ncps} only requires the information of $\nu_{\mbf{a}}$ on simple directed cycles, which we can compute with the information of $\nu_{\mbf{a}}^0$ on any test graph that can be obtained as a quotient of a simple directed cycle. One can easily verify that a double tree obtained as a quotient of a simple directed cycle can only have opposing twin edges \cite[Figure 5]{Au18a}. Since the pseudovariances $(\beta_i)_{i \in I}$ only appear in \eqref{eq:semicircular_traffic_family} alongside congruent twin edges, the limiting traffic distribution of $(\mbf{W}_N^{(i)})_{i \in I}$ for general $(\beta_i)_{i \in I} \subset \D$ still defines a semicircular family of covariance $\mbf{K}(i, i') = \indc{i = i'}$ in $(\mcal{A}, \varphi_\tau)$.

The preceding paragraph outlines the strategy we adopt in the proof of Corollary \ref{cor:permuted_nc}. After proving Theorem \ref{thm:permuted_traffic}, we will analyze the relevant contributions from the traffic distribution to the usual distribution. This restriction to test graphs arising as quotients of simple directed cycles explains the difference between the conditions in Theorem \ref{thm:permuted_traffic} and Corollary \ref{cor:permuted_nc}.

\section{Proofs of the main results}
\subsection{Asymptotics for the injective traffic distribution of $(\mbf{M}_N^{(i)})_{i \in I}$}
We prove Theorem \ref{thm:permuted_traffic} in steps. First, we show that staying off the grid allows us to restrict our attention to double trees.

\begin{lemma}\label{lem:double_tree}
Let $\mcal{M}_N = (\mbf{M}_N^{(i)})_{i \in I}$. Suppose that $(\sigma_N^{(i')})^{-1} \circ \sigma_N^{(i)}$ stays off the grid for every $i, i' \in I$. In that case, if $T \in \mcal{T}\langle I \rangle$ is not a double tree, then $\nu_{\mcal{M}_N}^0[T] = o_T(1)$.
\end{lemma}
\begin{proof}
For a test graph $T = (V, E, \gamma) \in \mcal{T}\langle I \rangle$, the injective traffic distribution can be written as
\begin{equation}\label{eq:injective_permuted_wigner}
\begin{aligned}
  \nu_{\mcal{M}_N}^0[T] &= \frac{1}{N} \sum_{\phi: V \hookrightarrow [N]} \E\Big[\prod_{e \in E} \mbf{M}_N^{(\gamma(e))}(\phi(e))\Big] \\
              &= \frac{1}{N} \sum_{\phi: V \hookrightarrow [N]} \E\Big[\prod_{e \in E} \mbf{W}_N(\sigma_N^{(\gamma(e))}(\phi(e)))\Big] \\
              &= \frac{1}{N^{1 + \frac{\#(E)}{2}}}\sum_{\phi: V \hookrightarrow [N]} \E\Big[\prod_{e \in E} \mbf{X}_N(\sigma_N^{(\gamma(e))}(\phi(e)))\Big].
\end{aligned}
\end{equation}
Our strong moment assumption \eqref{eq:moment_bound} implies that
\[
  \E\Big[\prod_{e \in E} \mbf{X}_N(\sigma_N^{(\gamma(e))}(\phi(e)))\Big] = O_T(1)
\]
uniformly in $N$ and $\phi$. We immediately obtain the elementary bound
\[
  \nu_{\mcal{M}_N}^0[T] = O_T(N^{\#(V) - (1 + \frac{\#(E)}{2})}),
\]
which is quite far from what we need. To this end, we define the sets 
\begin{align*}
  \wtilde{\mcal{N}}_1 &= \{[e] \in \wtilde{\mcal{N}} : m_{[e]} = 1\}; \\
  \wtilde{\mcal{N}}_{\geq 2} &= \{[e] \in \wtilde{\mcal{N}} : m_{[e]} \geq 2\}; \\
  \mcal{N}_1 &= \{e \in \mcal{N} : [e] \in \wtilde{\mcal{N}}_1\}; \\
  \mcal{N}_{\geq 2} &= \{e \in \mcal{N} : [e] \in \wtilde{\mcal{N}}_{\geq 2}\},
\end{align*}
where we recall that $E = \mcal{L} \sqcup \mcal{N}$. Note that
\begin{equation}\label{eq:simple_graph_edge_bounds}
\begin{aligned}
  \#(\mcal{N}_1) &= \#(\wtilde{\mcal{N}}_1); \\
  \#(\mcal{N}_{\geq 2}) &\geq  2\#(\wtilde{\mcal{N}}_{\geq 2}).
\end{aligned}
\end{equation}

Now, suppose that there exists an edge $e_0 \in \mcal{N}_1$. Since the elements $(\mbf{X}_N(j, k))_{1 \leq j < k \leq N}$ are independent and centered, the injectivity of the map $\phi$ implies that
\[
  \E\Big[\prod_{e \in E} \mbf{X}_N(\sigma_N^{(\gamma(e))}(\phi(e)))\Big] = 0
\]
unless there exists another edge $e_1 \in E$ such that
\begin{equation}\label{eq:singleton_meetup}
  \sigma_N^{(\gamma(e_1))}(\phi(e_1)) \in \{\sigma_N^{(\gamma(e_0))}(\phi(e_0)), \sigma_N^{(\gamma(e_0))}(\phi(e_0))^\intercal\}.
\end{equation}
Since $m_{[e_0]} = 1$, we know that $[e_1] \neq [e_0]$. Furthermore, since the permutations $(\sigma_N^{(i)})_{i \in I}$ are symmetric, it must be that $e_1 \in \mcal{N}$. Applying the inverse, \eqref{eq:singleton_meetup} is equivalent to
\begin{equation}\label{eq:expanded_singleton_meetup}
\begin{aligned}
  (\phi(\tar(e_1)), \phi(\src(e_1))) \in \{&(\sigma_N^{(\gamma(e_1))})^{-1} \circ \sigma_N^{(\gamma(e_0))}(\phi(\tar(e_0)), \phi(\src(e_0))), \\
  &(\sigma_N^{(\gamma(e_1))})^{-1} \circ \sigma_N^{(\gamma(e_0))}(\phi(\src(e_0)), \phi(\tar(e_0)))\}.
\end{aligned}
\end{equation}
In particular, $\{\phi(\tar(e_1)), \phi(\src(e_1))\}$ is completely determined by $\{\phi(\tar(e_0)), \phi(\src(e_0))\}$ and $\{\gamma(e_0), \gamma(e_1)\}$. This means we no longer have the freedom to choose the value of $\phi$ on $\{\tar(e_1), \src(e_1)\}$. Since $[e_1] \neq [e_0]$, we know that
\[
  \op{overlap}([e_0], [e_1]) := \#\big(\{\tar(e_0), \src(e_0)\} \cap \{\tar(e_1), \src(e_1)\}\big) \leq 1 .
\]
The singleton edge $[e_0]$ then reduces the number of degrees of freedom in choosing our map $\phi$ by
\[
  2 - \op{overlap}([e_0], [e_1]) \geq 1.
\]
This is true of every singleton edge in $\underline{G} = (V, \wtilde{\mcal{N}})$; however, it could be that the matching edge $[e_1]$ is also a singleton $m_{[e_1]} = 1$.

Let
\[
  A_N = \left\{\phi: V \hookrightarrow [N] : \E\Big[\prod_{e \in E} \mbf{X}_N(\sigma_N^{(\gamma(e))}(\phi(e)))\Big] \neq 0\right\}.
\]
We refine the obvious bound $\#(A_N) = O(N^{\#(V)})$ by exploring the underlying simple graph $\underline{G} = (V, \wtilde{\mcal{N}})$ with the restriction from the previous paragraph in mind. We keep track of the vertices (resp., edges) we have encountered so far in our exploration with the set $V_{\op{enc}}$ (resp., $\wtilde{\mcal{N}}_{\op{enc}}$), updating it as we go. We think of $V_{\op{enc}}$ as the set of vertices for which we have already assigned the values $\phi|_{V_{\op{enc}}}$. Of course, since $T$ is connected, we know that $\underline{G}$ is still connected. We start with $(V_{\op{enc}}, \wtilde{\mcal{N}}_{\op{enc}}) = (\emptyset, \emptyset)$ and choose an arbitrary initial vertex $v_0 \in V$, for which there are at most $N$ choices of $\phi(v_0) \in [N]$. Having made this choice, we update $V_{\op{enc}} = V_{\op{enc}} \cup \{v_0\}$.

At any given stage, we have a list of encountered vertices $V_{\op{enc}}$ and encountered edges $\wtilde{\mcal{N}}_{\op{enc}}$. If there are no remaining edges $\wtilde{\mcal{N}} \setminus \wtilde{\mcal{N}}_{\op{enc}} = \emptyset$, then we stop. Otherwise, there is an edge $[e] \in \wtilde{\mcal{N}}\setminus\wtilde{\mcal{N}}_{\op{enc}}$ incident to a vertex $v \in V_{\op{enc}}$. Let $w \in \{\src(e), \tar(e)\}\setminus\{v\}$ be the other vertex incident to this edge. Again, we have at most $N$ choices for $\phi(w)$; however, it could be that we already encountered $w$ earlier (for example, $\underline{G}$ is not necessarily a tree). In the latter case, the choice of $\phi(w)$ has already been made and we do not get to choose again. In any event, we update $V_{\op{enc}} = V_{\op{enc}} \cup \{w\}$ and $\wtilde{\mcal{N}}_{\op{enc}} = \wtilde{\mcal{N}}_{\op{enc}} \cup \{[e]\}$. If $[e]$ is further a singleton edge, then we know that $\phi$ matches $[e]$ to another edge $[e']$ as in \eqref{eq:expanded_singleton_meetup} and we lose the freedom to choose the value(s) $\phi|_{\{\tar(e'),\src(e')\}\setminus\{\tar(e),\src(e)\}}$. Of course, it could be that $(\{\tar(e'),\src(e')\}\setminus\{\tar(e),\src(e)\}) \cap V_{\op{enc}} \neq \emptyset$. In any event, for any edge $e'$ matched to $e$ by $\phi$ (possibly more than one), we update $V_{\op{enc}} = V_{\op{enc}} \cup \{\tar(e'),\src(e')\}$ and $\wtilde{\mcal{N}}_{\op{enc}} = \wtilde{\mcal{N}}_{\op{enc}} \cup \{[e']\}$ and start the procedure again.

The procedure stops once we have explored all of the edges $\wtilde{\mcal{N}}_{\op{enc}} = \wtilde{\mcal{N}}$. Since $\underline{G}$ is connected, this means that we will have explored all of the vertices $V_{\op{enc}} = V$. In other words, we will have constructed a function $\phi: V \hookrightarrow [N]$. To bound the number of possible functions constructed in this manner, note that every exploration of an edge contributes a factor of at most $N$; however, if the edge happens to be a singleton edge, it automatically explores at least one other edge at no cost (i.e., we do not count this induced exploration toward the runtime). In the worst case scenario (i.e., longest total exploration time), a singleton edge explores exactly one other edge, which itself is also a singleton edge, using up an edge that would have given us a free exploration. Counting the runtime of the procedure then gives the bound
\begin{equation}\label{eq:runtime_function_bound}
  \#(A_N) = O_T(N^{1 + \#(\wtilde{\mcal{N}}_{\geq 2}) + \frac{\#(\wtilde{\mcal{N}}_1)}{2}}).
\end{equation}
Of course, this is only a refinement of our earlier bound if $1 + \#(\wtilde{\mcal{N}}_{\geq 2}) + \frac{\#(\wtilde{\mcal{N}}_1)}{2} < \#(V)$, which is not true in general. Nevertheless, putting the runtime bound \eqref{eq:runtime_function_bound} back into the injective traffic distribution \eqref{eq:injective_permuted_wigner}, we have the asymptotic
\[
  \nu_{\mcal{M}_N}^0[T] = O_T(N^{\#(\wtilde{\mcal{N}}_{\geq 2}) + \frac{\#(\wtilde{\mcal{N}}_1)}{2} - \frac{\#(E)}{2}}).
\]
We can rewrite the exponent as
\begin{align*}
  \#(\wtilde{\mcal{N}}_{\geq 2}) + \frac{\#(\wtilde{\mcal{N}}_1)}{2} - \frac{\#(E)}{2} &= \#(\wtilde{\mcal{N}}_{\geq 2}) + \frac{\#(\wtilde{\mcal{N}}_1)}{2}- \frac{\#(\mcal{N}_1) + \#(\mcal{N}_{\geq 2}) + \#(\mcal{L})}{2} \\
                                                                     &= \frac{2\#(\wtilde{\mcal{N}}_{\geq 2}) - \#(\mcal{N}_{\geq 2})}{2} - \frac{\#(\mcal{L})}{2}.
\end{align*}
So, in view of \eqref{eq:simple_graph_edge_bounds}, we can already assume that $\mcal{L} = \emptyset$ and $\#(\mcal{N}_{\geq 2}) = 2\#(\wtilde{\mcal{N}}_{\geq 2})$, the last equality implying that $m_{[e]} \in \{1, 2\}$ for every $[e] \in \wtilde{\mcal{N}}$.

Even so, we see that the injective traffic distribution is only just surviving the normalization: anything strictly less than runtime bound would result in $\nu_{\mcal{M}_N}^0[T] = o_T(1)$. In fact, the runtime bound already implicitly assumes some optimality in the matching between singleton edges, which we explain now. Suppose that there are singleton edges $[e]$ and $[e']$ matched by $\phi$ in the exploration procedure. For concreteness, assume that $[e]$ is the edge that we explore and $[e']$ is the edge that gets matched to $[e]$ as a result of the free exploration induced by $[e]$. Let $A_{N, \phi} \subset A_N$ be the subset of vertex labelings with the same matching of the edges as $\phi$. Note that it is possible for $A_{N, \phi} = A_{N, \phi'}$ with $\phi \neq \phi'$. In particular, there are only finitely many different matchings of the edges: this number only depends on the finite graph $T$ and is independent of $N$.

Since $\underline{G}$ is connected, we know that there exists a path $([e_1], \ldots, [e_k])$ in $\underline{G}$ with $[e_1] = [e]$ and $[e_k] = [e']$: by a path, we mean a walk with no repetition of vertices. We claim that for any such path, there must exist singleton edges $[e_j]$ and $[e_{j+1}]$ along this path that are matched by $\phi$ unless $\#(A_{N, \phi}) = o_T(N^{\#(\wtilde{\mcal{N}}_{\geq 2}) + \frac{\#(\wtilde{\mcal{N}}_1)}{2} - \frac{\#(E)}{2}})$. In other words, there must exist adjacent singleton edges along this path that get matched if the injective traffic distribution is to survive the normalization. We prove this by induction on the length $k$ of the path. The base case $k = 2$ is trivial since $[e_1]$ and $[e_2]$ will be adjacent by definition. Suppose that $k > 2$ and we know the result for all shorter lengths. In that case, $[e_1]$ and $[e_k]$ are not adjacent, meaning 
\[
  \op{overlap}([e_1], [e_k]) = 0. 
\]
But then the choice of $\phi|_{\{\tar(e_1), \src(e_1)\}}$ determines the value of $\phi$ on \emph{two} other vertices $\phi|_{\{\tar(e_k)), \src(e_k)\}}$. We choose to bound $\#(A_{N, \phi})$ by starting the exploration procedure at a vertex incident to $[e_1]$ and immediately getting matched to $[e_k]$. We can then explore the intermediate path $([e_2], \ldots, [e_{k-1}])$; however, by the time we explore $[e_{k-1}]$, we do not get to choose the value of the vertex we encounter at this stage since it was already encountered earlier by the free exploration coming from $[e_k]$. This would drop a factor of $N$ from our runtime bound \emph{unless} the exploration of $[e_{k-1}]$ was a free exploration coming from another singleton edge $[e'']$. Furthermore, the optimality of the runtime bound says that a singleton edge should only get matched to exactly one other (singleton) edge, so $[e_{k-1}]$ itself must be a singleton edge. Since we have only explored the edges along the path $([e_1], \ldots, [e_k])$, it must be that $[e''] = [e_{k''}]$ for some $1 < k'' < k-1$. In that case, $([e_{k''}], \ldots, [e_{k-1}])$ is a shorter path with singleton edges $[e_{k''}]$ and $[e_{k-1}]$ matched by $\phi$ and we can apply the induction hypothesis. 

We use this to modify our test graph to keep track of the dependences on the vertex labelings $\phi$ induced by the matchings. If $[e]$ and $[e']$ are adjacent singleton edges that are matched by $\phi$, then we can record the reduction in the degrees of freedom in choosing $\phi$ by identifying the vertices $\{\tar(e'),\src(e')\}\setminus\{\tar(e),\src(e)\}$ and $\{\tar(e),\src(e)\}\setminus\{\tar(e'),\src(e')\}$ (we can think of pinching the edges $[e]$ and $[e']$ together). From the point of view of our refined runtime bound above, this removes a pair of adjacent singleton edges along the path from $[e_1]$ to $[e_k]$. We can apply the same reasoning to this new test graph to continue pinching possibly newly adjacent singleton edges matched by $\phi$ until there are no more singleton edges. See Figure \ref{fig:pinching} for an illustration.

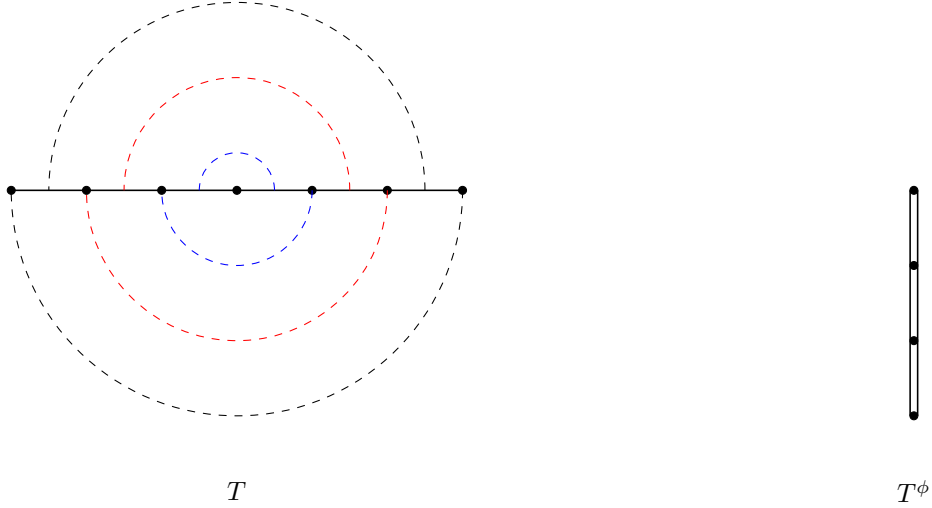
\begin{figure}
  \begin{tikzpicture}[baseline=(current  bounding  box.center)]
    \node at (3, -4) {$T$};
    
    \draw[fill=black] (0, 0) circle (1.5pt);
    \draw[fill=black] (1,0) circle (1.5pt);
    \draw[fill=black] (2, 0) circle (1.5pt);
    \draw[fill=black] (3, 0) circle (1.5pt);
    \draw[fill=black] (4, 0) circle (1.5pt);
    \draw[fill=black] (5, 0) circle (1.5pt);
    \draw[fill=black] (6, 0) circle (1.5pt);

    \draw[semithick, -] (0, 0) to (1, 0);
    \draw[semithick, -] (1, 0) to (2, 0);
    \draw[semithick, -] (2, 0) to (3, 0);
    \draw[semithick, -] (3, 0) to (4, 0);
    \draw[semithick, -] (4, 0) to (5, 0);
    \draw[semithick, -] (5, 0) to (6, 0);

    \draw[dashed] (5.5,0) arc (0:180:2.5);
    \draw[red, dashed] (4.5,0) arc (0:180:1.5);
    \draw[blue, dashed] (3.5,0) arc (0:180:.5);

    \draw[dashed] (6,0) arc (0:-180:3);
    \draw[red, dashed] (5,0) arc (0:-180:2);
    \draw[blue, dashed] (4,0) arc (0:-180:1);

    \node at (12, -4) {$T^\phi$};
    
    \draw[fill=black] (12, 0) circle (1.5pt);
    \draw[fill=black] (12, -1) circle (1.5pt);
    \draw[fill=black] (12, -2) circle (1.5pt);
    \draw[fill=black] (12, -3) circle (1.5pt);

    \draw[semithick, -] (11.95, 0) to (11.95, -1);
    \draw[semithick, -] (12.05, 0) to (12.05, -1);
    \draw[semithick, -] (11.95, -1) to (11.95, -2);
    \draw[semithick, -] (12.05, -1) to (12.05, -2);
    \draw[semithick, -] (11.95, -2) to (11.95, -3);
    \draw[semithick, -] (12.05, -2) to (12.05, -3);

  \end{tikzpicture}
\caption{An example of the successive pinching of adjacent singleton edges starting with $T$ and resulting in $T^\phi$. The dashed lines above the edges indicate the edge matchings induced by $\phi$. The dashed lines below the edges indicate the corresponding vertices that are identified as a result of the pinching. For simplicity, the test graphs are drawn with undirected edges.}\label{fig:pinching}
\end{figure}

We explain the intuition behind this result. A matching between singleton edges that are not adjacent is costly because it explores an additional two vertices. If we incur such a matching, we need to proceed optimally to meet the upper bound of the runtime procedure. This can only be achieved if we use free explorations to bridge between the non-adjacent singleton edges. If we think of building a bridge by laying down edges from both sides, then the adjacent singleton edges that get matched are where the two sides meet.

Now, assume that $\phi$ is such an optimal matching. Let $T^\phi = (V^\phi, E^\phi, \gamma^\phi)$ be the test graph obtained from $T$ by pinching all the adjacent singleton edges matched by $\phi$ iteratively. Formally, we obtain $T^\phi$ from $T$ as follows. First, we pinch the adjacent singleton edges matched by $\phi$. Note that this might create newly adjacent singleton edges that are matched by $\phi$, in which case we repeat the pinching process. Our analysis above shows that we will eventually stop, at which point we no longer have any singleton edges. Of course, $\#(E^\phi) = \#(E)$, and we can rephrase our refined runtime bound as $\#(A_{N, \phi}) = O_T(N^{\#(V^\phi)})$.

We claim that $\nu_{\mcal{M}_N}^0[T] = o_T(1)$ unless there exists a matching induced by a vertex labeling $\phi$ such that $T^\phi$ is a double tree. Indeed, by our earlier work, we already know that $\nu_{\mcal{M}_N}^0[T] = o_T(1)$ unless $\mcal{L}^\phi = \emptyset$ and $\#(\mcal{N}_{\geq 2}^\phi) = 2\#(\wtilde{\mcal{N}}_{\geq 2}^\phi)$, where the superscript corresponds to the constructions from before, but now for the test graph $T^\phi$. The pinching implies that we no longer have any singleton edges, so $\wtilde{\mcal{N}}_1^\phi = \emptyset$. Under these conditions, $T^\phi$ is a double tree iff $\#(V^\phi) = \#(\wtilde{\mcal{N}}^\phi) + 1$. To finish, assume that $\#(V^\phi) < \#(\wtilde{\mcal{N}}^\phi) + 1$. In that case,
\[
  \#(A_{N, \phi}) = O_T(N^{\#(V^\phi)}) = O_T(N^{\#(\wtilde{\mcal{N}}^\phi)})
\]
and the contributions to $\nu_{\mcal{M}_N}^0[T]$ from this class of vertex labelings is 
\[
  O_T(N^{\#(\wtilde{\mcal{N}}^\phi) - 1 - \frac{\#(E)}{2}}) = O_T(N^{\#(\wtilde{\mcal{N}}^\phi) - 1 - \frac{\#(E^\phi)}{2}}) = O_T(N^{-1}).
\]
Since there are only a finite class of matchings of $T$, the claim follows.

We summarize our progress so far:
\[
  \nu_{\mcal{M}_N}^0[T] = o_T(1)
\]
unless $\mcal{L} = \emptyset$, $\#(\mcal{N}_{\geq 2}) = \#(\mcal{N}_2)$, and there exists a matching of singleton edges such that the resulting pinched test graph is a double tree. In particular, if there are no singleton edges to pinch, then the test graph is already a double tree. Moreover, the pinching involves a sequence of successively adjacent singleton edges. We finish by using the grid condition to rule out the possibility of pinching.

If $[e]$ and $[e']$ are adjacent singleton edges matched by $\phi$, then we can write the condition in \eqref{eq:expanded_singleton_meetup} for this pair as
\begin{equation}\label{eq:adjacent_singleton_matching}
  (\phi(\tar(e)), \phi(\src(e))) \in \op{Grid}((\sigma_N^{\gamma(e')})^{-1} \circ \sigma_N^{\gamma(e)}).
\end{equation}
By assumption, $(\sigma_N^{(i')})^{-1} \circ \sigma_N^{(i)}$ stays off the grid for every $i, i' \in I$, which implies that
\[
  \#(A_{N, \phi}) = o_T(N^{1 + \#(\wtilde{\mcal{N}}_{\geq 2}) + \frac{\#(\wtilde{\mcal{N}}_1)}{2}}).
\]
Indeed, we can start our procedure at a vertex incident to $[e]$ and immediately get matched to $[e']$: staying off the grid implies that the total cost at this point is $o(N^2)$ as opposed to our earlier $O(N^2)$. We conclude that
\[
  \nu_{\mcal{M}_N}^0[T] = o_T(1)
\]
unless there are no singleton edges to match, as was to be shown.
\end{proof}

We see that staying off the grid restricts the possible support of our limiting injective traffic distribution to double trees. The actual calculation on double trees requires our other assumption about the homogeneity of the fixed points (resp., transposed points).

\begin{lemma}\label{lem:homogeneity_traffic}
Suppose that we have the limits
\begin{equation}\label{eq:lemma_proportion_limits}
\begin{aligned}
  \lim_{N \to \infty} \max_{j \in [N]} \frac{\#\Big(\text{\emph{FP}}\big((\sigma_N^{(i')})^{-1} \circ \sigma_N^{(i)}, j\big)\Big)}{N} = \lim_{N \to \infty} \min_{j \in [N]} \frac{\#\Big(\text{\emph{FP}}\big((\sigma_N^{(i')})^{-1} \circ \sigma_N^{(i)}, j\big)\Big)}{N} &= a_{i, i'}; \\
  \lim_{N \to \infty} \max_{j \in [N]} \frac{\#\Big(\text{\emph{TP}}\big((\sigma_N^{(i')})^{-1} \circ \sigma_N^{(i)}, j\big)\Big)}{N} = \lim_{N \to \infty} \min_{j \in [N]} \frac{\#\Big(\text{\emph{TP}}\big((\sigma_N^{(i')})^{-1} \circ \sigma_N^{(i)}, j\big)\Big)}{N} &= b_{i, i'}.
\end{aligned}
\end{equation}
In that case, if $T \in \mcal{T}\langle I \rangle$ is a double tree, then
\[
\begin{aligned}
  \lim_{N \to \infty} \nu_{\mcal{M}_N}^0[T] = \prod_{[e] = \{e, e'\} \in \wtilde{\mcal{N}}} \Big(&\indc{\{e, e'\} \text{ \emph{are opposing}}}(a_{\gamma(e), \gamma(e')}+ b_{\gamma(e), \gamma(e')}\beta) \\
  + &\indc{\{e, e'\} \text{ \emph{are congruent}}}(a_{\gamma(e), \gamma(e')}\beta + b_{\gamma(e), \gamma(e')})\Big).
\end{aligned}
\]
\end{lemma}
\begin{proof}
Assume that $T = (V, E, \gamma) \in \mcal{T}\langle I \rangle$ is a double tree. We rewrite the injective traffic state accordingly:
\[
  \nu_{\mcal{M}_N}^0[T] = \frac{1}{N^{\#(\wtilde{\mcal{N}}) + 1}}\sum_{\phi: V \hookrightarrow [N]} \E\Big[\prod_{[e] = \{e, e'\} \in \wtilde{\mcal{N}}} \big[\mbf{X}_N(\sigma_N^{(\gamma(e))}(\phi(e)))\mbf{X}_N(\sigma_N^{(\gamma(e'))}(\phi(e')))\big]\Big].
\]
We would like to use the independence of the $(\mbf{X}_N(j, k))_{1 \leq j < k \leq N}$ and the injectivity of $\phi$ to factor the expectation; however, the map $\phi$ could possibly match different edges as in \eqref{eq:expanded_singleton_meetup}. Fortunately, our earlier work shows that this is typically not the case. In particular, matching edges reduces the number of degrees of freedom, and we only have just enough to survive the normalization as it is since $\#(V) = \#(\mcal{N}) + 1$. Using our strong moment assumption \eqref{eq:moment_bound}, we see that the contribution from any given summand is bounded, which allows us to pass to the desired factorization at negligible cost:
\[
  \nu_{\mcal{M}_N}^0[T] = \frac{1}{N^{\#(V)}}\sum_{\phi: V \hookrightarrow [N]} \prod_{[e] = \{e, e'\} \in \wtilde{\mcal{N}}} \E\Big[\mbf{X}_N(\sigma_N^{(\gamma(e))}(\phi(e)))\mbf{X}_N(\sigma_N^{(\gamma(e'))}(\phi(e')))\Big] + o_T(1).
\]
Similarly, since most maps from $V$ to $[N]$ are injective ($\lim_{N \to \infty} \frac{N^{\underline{\#(V)}}}{N^{\#(V)}} = 1$), we can further write
\[
  \nu_{\mcal{M}_N}^0[T] = \frac{1}{N^{\#(V)}}\sum_{\phi(v_{\#(V)}) \in [N]} \cdots \sum_{\phi(v_1) \in [N]} \prod_{[e] = \{e, e'\} \in \wtilde{\mcal{N}}} \E\Big[\mbf{X}_N(\sigma_N^{(\gamma(e))}(\phi(e)))\mbf{X}_N(\sigma_N^{(\gamma(e'))}(\phi(e')))\Big] + o_T(1)
\]
for any enumeration $v_1, \ldots, v_{\#(V)}$ of the vertices.

We calculate the possible contributions from a pair of twin edges:\small
\[
  \E\Big[\mbf{X}_N(\sigma_N^{(\gamma(e))}(\phi(e)))\mbf{X}_N(\sigma_N^{(\gamma(e'))}(\phi(e')))\Big] =
  \begin{dcases}
    1 & \text{if } \{e, e'\} \text{ are opposing and } \phi(e) \in \op{FP}((\sigma_N^{(\gamma(e'))})^{-1} \circ \sigma_N^{(\gamma(e))}) \\
    & \text{or if }  \{e, e'\} \text{ are congruent and }  \phi(e) \in \op{TP}((\sigma_N^{(\gamma(e'))})^{-1} \circ \sigma_N^{(\gamma(e))}); \\
    \beta & \text{if } \{e, e'\} \text{ are opposing and } \phi(e) \in \op{TP}((\sigma_N^{(\gamma(e'))})^{-1} \circ \sigma_N^{(\gamma(e))}) \\
    & \text{or if }  \{e, e'\} \text{ are congruent and }  \phi(e) \in \op{FP}((\sigma_N^{(\gamma(e'))})^{-1} \circ \sigma_N^{(\gamma(e))}); \\  
    0 &\text{else.}
  \end{dcases}
\]\normalsize
Of course, we need to know the value of $\phi(e) = (\phi(\tar(e)), \phi(\src(e))$ to make this determination, but the limits in \eqref{eq:lemma_proportion_limits} imply that the average contribution over all possible values of $\phi(\src(e))$ is asymptotically independent of the choice of $\phi(\tar(e))$ (and vice versa). We can then hope to evaluate the large $N$ limit of our injective traffic distribution by interpreting the normalized iterated sum as a successive chain of conditional expectations. 

To achieve this, we consider the following enumeration of the vertices. First, we enumerate the leaves of our double tree $v_1, \ldots, v_{n_1}$. We then remove these vertices and all incident edges, resulting in a new double tree with strictly fewer vertices. We repeat the procedure until all vertices have been accounted for $v_1, \ldots, v_{n_1}, v_{n_1 + 1}, \ldots, v_{n_2}, \ldots, v_{\#(V)}$. Using this iterative stripping process and the limits in \eqref{eq:lemma_proportion_limits}, we see that\footnotesize
\begin{equation}\label{eq:stripping_double_tree_limit}
\begin{aligned}
  &\lim_{N \to \infty} \frac{1}{N^{\#(V)}}\sum_{\phi(v_{\#(V)}) \in [N]} \sum_{\phi(v_{\#(V)-1}) \in [N]} \cdots \sum_{\phi(v_1) \in [N]} \prod_{[e] = \{e, e'\} \in \wtilde{\mcal{N}}} \E\Big[\mbf{X}_N(\sigma_N^{(\gamma(e))}(\phi(e)))\mbf{X}_N(\sigma_N^{(\gamma(e'))}(\phi(e')))\Big] \\
  = &\lim_{N \to \infty} \frac{1}{N}\sum_{\phi(v_{\#(V)}) \in [N]} \frac{1}{N} \sum_{\phi(v_{\#(V)-1}) \in [N]} \cdots \frac{1}{N}\sum_{\phi(v_1) \in [N]} \prod_{[e] = \{e, e'\} \in \wtilde{\mcal{N}}} \E\Big[\mbf{X}_N(\sigma_N^{(\gamma(e))}(\phi(e)))\mbf{X}_N(\sigma_N^{(\gamma(e'))}(\phi(e')))\Big] \\
  = &\prod_{[e] = \{e, e'\} \in \wtilde{\mcal{N}}} \Big(\indc{\{e, e'\} \text{ are opposing}}(a_{\gamma(e), \gamma(e')}+ b_{\gamma(e), \gamma(e')}\beta) \\
   &\phantom{\prod_{[e] = \{e, e'\} \in \wtilde{\mcal{N}}} \Big(}\hspace*{-10pt}+\indc{\{e, e'\} \text{ are congruent}}(a_{\gamma(e), \gamma(e')}\beta + b_{\gamma(e), \gamma(e')})\Big),
\end{aligned}
\end{equation}\normalsize
as was to be shown.
\end{proof}

\begin{rem}\label{rem:max_min_limits_traffic}
If $\beta \in (-1, 1)$, the limits in \eqref{eq:lemma_proportion_limits} are equivalent to
\begin{align*}
  &\lim_{N \to \infty} \max_{j \in [N]} \frac{\#\Big(\text{FP}\big((\sigma_N^{(i')})^{-1} \circ \sigma_N^{(i)}, j\big)\Big) + \#\Big(\text{TP}\big((\sigma_N^{(i')})^{-1} \circ \sigma_N^{(i)}, j\big)\Big)\beta}{N} \\
  = &\lim_{N \to \infty} \min_{j \in [N]} \frac{\#\Big(\text{FP}\big((\sigma_N^{(i')})^{-1} \circ \sigma_N^{(i)}, j\big)\Big) + \#\Big(\text{TP}\big((\sigma_N^{(i')})^{-1} \circ \sigma_N^{(i)}, j\big)\Big)\beta}{N} = a_{i, i'} + b_{i, i'}\beta
\end{align*}
and
\begin{align*}
  &\lim_{N \to \infty} \max_{j \in [N]} \frac{\#\Big(\text{FP}\big((\sigma_N^{(i')})^{-1} \circ \sigma_N^{(i)}, j\big)\Big)\beta + \#\Big(\text{TP}\big((\sigma_N^{(i')})^{-1} \circ \sigma_N^{(i)}, j\big)\Big)}{N} \\
  = &\lim_{N \to \infty} \min_{j \in [N]} \frac{\#\Big(\text{FP}\big((\sigma_N^{(i')})^{-1} \circ \sigma_N^{(i)}, j\big)\Big)\beta + \#\Big(\text{TP}\big((\sigma_N^{(i')})^{-1} \circ \sigma_N^{(i)}, j\big)\Big)}{N} = a_{i, i'}\beta + b_{i, i'}.
\end{align*}
In the case of $\beta \in \{-1, 1\}$, either of the two sets of limits above already implies the other, which is all we need to justify the convergence in \eqref{eq:stripping_double_tree_limit}. This explains the weaker condition replacing conditions \ref{cond:fp} and \ref{cond:tp} in Theorem \ref{thm:permuted_traffic}.
\end{rem}

This proves the first half of Theorem \ref{thm:permuted_traffic}. For the second half, we start with the converse of Lemma \ref{lem:double_tree}.
\begin{lemma}\label{lem:double_tree_converse}
Suppose that $\nu_{\mcal{M}_N}^0[T] = o_T(1)$ if $T \in \mcal{T}\langle I \rangle$ is not a double tree. In that case, if $\beta \in (-1, 1]$, then $(\sigma_N^{(i')})^{-1} \circ \sigma_N^{(i)}$ stays off the grid for every $i, i' \in I$.
\end{lemma}
\begin{proof}
By assumption,
\[
  \begin{tikzpicture}
    \node at (-.775, 0) {$\displaystyle\nu_{\mcal{M}_N}^0\Bigg[$};
    \node at (8.5, 0) {$\displaystyle\Bigg] = \frac{1}{N^2}\sum_{\phi: \{v_1, v_2, v_3\} \hookrightarrow [N]} \E\big[\mbf{X}_N(\sigma_N^{(i)}(\phi(v_1), \phi(v_2)))\mbf{X}_N(\sigma_N^{(i')}(\phi(v_2), \phi(v_3)))\big] = o_T(1)$,};
    \draw[fill=black] (0,0) circle (1.75pt);
    \node at (0, -.375) {\footnotesize$v_1$\normalsize};
    \draw[fill=black] (1,0) circle (1.75pt);
    \node at (1, -.375) {\footnotesize$v_2$\normalsize};
    \draw[fill=black] (2,0) circle (1.75pt);
    \node at (2, -.375) {\footnotesize$v_3$\normalsize};
    \draw[semithick, ->] (.875, 0) to node[midway, above] {\footnotesize$i$\normalsize} (.125, 0);
    \draw[semithick, ->] (1.875, 0) to node[midway, above] {\footnotesize$i'$\normalsize} (1.125, 0);
  \end{tikzpicture}  
\]
where
\begin{align*}
  \E\big[\mbf{X}_N(\sigma_N^{(i)}(\phi(v_1), \phi(v_2)))\mbf{X}_N(\sigma_N^{(i')}(\phi(v_2), \phi(v_3)))\big] &= \E\big[\mbf{X}_N((\sigma_N^{(i')})^{-1} \circ \sigma_N^{(i)}(\phi(v_1), \phi(v_2)))\mbf{X}_N(\phi(v_2), \phi(v_3))\big] \\
  &= \begin{dcases}
    1 &\text{if } (\sigma_N^{(i')})^{-1} \circ \sigma_N^{(i)}(\phi(v_1), \phi(v_2)) = (\phi(v_3), \phi(v_2)); \\
    \beta &\text{if } (\sigma_N^{(i')})^{-1} \circ \sigma_N^{(i)}(\phi(v_1), \phi(v_2)) = (\phi(v_2), \phi(v_3)); \\
    0 &\text{else.}
    \end{dcases}
\end{align*}
Similarly,
\[
  \begin{tikzpicture}
    \node at (-.775, 0) {$\displaystyle\nu_{\mcal{M}_N}^0\Bigg[$};
    \node at (8.5, 0) {$\displaystyle\Bigg] = \frac{1}{N^2}\sum_{\phi: \{v_1, v_2, v_3\} \hookrightarrow [N]} \E\big[\mbf{X}_N(\sigma_N^{(i)}(\phi(v_1), \phi(v_2)))\mbf{X}_N(\sigma_N^{(i')}(\phi(v_3), \phi(v_2)))\big] = o_T(1)$,};
    \draw[fill=black] (0,0) circle (1.75pt);
    \node at (0, -.375) {\footnotesize$v_1$\normalsize};
    \draw[fill=black] (1,0) circle (1.75pt);
    \node at (1, -.375) {\footnotesize$v_2$\normalsize};
    \draw[fill=black] (2,0) circle (1.75pt);
    \node at (2, -.375) {\footnotesize$v_3$\normalsize};
    \draw[semithick, ->] (.875, 0) to node[midway, above] {\footnotesize$i$\normalsize} (.125, 0);
    \draw[semithick, ->] (1.125, 0) to node[midway, above] {\footnotesize$i'$\normalsize} (1.875, 0);
  \end{tikzpicture}  
\]
where
\begin{align*}
  \E\big[\mbf{X}_N(\sigma_N^{(i)}(\phi(v_1), \phi(v_2)))\mbf{X}_N(\sigma_N^{(i')}(\phi(v_3), \phi(v_2)))\big] &= \E\big[\mbf{X}_N((\sigma_N^{(i')})^{-1} \circ \sigma_N^{(i)}(\phi(v_1), \phi(v_2)))\mbf{X}_N(\phi(v_3), \phi(v_2))\big] \\
  &= \begin{dcases}
    \beta &\text{if } (\sigma_N^{(i')})^{-1} \circ \sigma_N^{(i)}(\phi(v_1), \phi(v_2)) = (\phi(v_3), \phi(v_2)); \\
    1 &\text{if } (\sigma_N^{(i')})^{-1} \circ \sigma_N^{(i)}(\phi(v_1), \phi(v_2)) = (\phi(v_2), \phi(v_3)); \\
    0 &\text{else.}
    \end{dcases}
\end{align*}
In other words, the sets
\begin{align*}
  \Gamma_N\big((\sigma_N^{(i')})^{-1} \circ \sigma_N^{(i)}\big) &:= \{(j, k, l) \in [N]^{\underline{3}}:  (\sigma_N^{(i')})^{-1} \circ \sigma_N^{(i)}(j, k) = (l, k)\}; \\
  \chi_N\big((\sigma_N^{(i')})^{-1} \circ \sigma_N^{(i)}\big) &:= \{(j, k, l) \in [N]^{\underline{3}}:  (\sigma_N^{(i')})^{-1} \circ \sigma_N^{(i)}(j, k) = (k, l)\};
\end{align*}
satisfy the asymptotics
\begin{equation}\label{eq:asymptotics_grid}
\begin{aligned}
  \#\Big(\Gamma_N\big((\sigma_N^{(i')})^{-1} \circ \sigma_N^{(i)}\big)\Big) + \beta\#\Big(\chi_N\big((\sigma_N^{(i')})^{-1} \circ \sigma_N^{(i)}\big)\Big) = o(N^2); \\
  \beta\#\Big(\Gamma_N\big((\sigma_N^{(i')})^{-1} \circ \sigma_N^{(i)}\big)\Big) + \#\Big(\chi_N\big((\sigma_N^{(i')})^{-1} \circ \sigma_N^{(i)}\big)\Big) = o(N^2).
\end{aligned}
\end{equation}
Since $\beta \in (-1, 1]$, this implies
\[
  \#\Big(\chi_N\big((\sigma_N^{(i')})^{-1} \circ \sigma_N^{(i)}\big)\Big) + \#\Big(\Gamma_N\big((\sigma_N^{(i')})^{-1} \circ \sigma_N^{(i)}\big)\Big) = o(N^2).
\]
The symmetry of the permutations then implies that
\[
  \#\Big(\text{Grid}\big((\sigma_N^{(i')})^{-1} \circ \sigma_N^{(i)}\big)\Big) = o(N^2),
\]
as was to be shown.
\end{proof}
\begin{rem}\label{rem:perfect_cancellation_exceptional_beta}
In the case of $\beta = -1$, the asymptotics in \eqref{eq:asymptotics_grid} only allow us to conclude that\small
\[
  \lim_{N \to \infty} \frac{\#\big(\{(j, k, l) \in [N]^3: ((\sigma_N^{(i')})^{-1} \circ \sigma_N^{(i)})(j, k) = (l, k)\}\big) - \#\big(\{(j, k, l) \in [N]^3: ((\sigma_N^{(i')})^{-1} \circ \sigma_N^{(i)})(j, k) = (k, l)\}\big)}{N^2} = 0,
\]\normalsize
which is the weaker condition replacing condition \ref{Cond:grid} in Theorem \ref{thm:permuted_traffic}. Applying the Cauchy-Schwarz inequality to the asymptotic from the the non-double tree
\[
  \begin{tikzpicture}
    \node at (-.5, 0) {$\displaystyle\nu_{\mcal{M}_N}^0\Bigg[$};
    \node at (2.825, 0) {$\displaystyle\Bigg] = o_T(1)$,};
    \draw[fill=black] (.25,0) circle (1.75pt);
    \draw[fill=black] (1,0) circle (1.75pt);
    \draw[fill=black] (1,.75) circle (1.75pt);
    \draw[fill=black] (1,-.75) circle (1.75pt);
    \draw[fill=black] (1.75,0) circle (1.75pt);
    \draw[semithick, ->] (.875, 0) to node[midway, below] {\footnotesize$i$\normalsize} (.375, 0);
    \draw[semithick, ->] (1.625, 0) to node[midway, above] {\footnotesize$i'$\normalsize} (1.125, 0);
    \draw[semithick, ->] (1, -.625) to node[midway, right] {\footnotesize$i'$\normalsize} (1, -.125);
    \draw[semithick, ->] (1, .125) to node[midway, left] {\footnotesize$i$\normalsize} (1, .625);
  \end{tikzpicture}  
\]
one can upgrade this to\footnotesize
\[
  \sum_{j \in [N]} \left|\#\big(\{(k, l) \in [N]^2: ((\sigma_N^{(i')})^{-1} \circ \sigma_N^{(i)})(j, k) = (l, k)\}\big) - \#\big(\{(k, l) \in [N]^3: ((\sigma_N^{(i')})^{-1} \circ \sigma_N^{(i)})(j, k) = (k, l)\}\big)\right| = o_{i, i'}(N^2).
\]\normalsize
This seems like the most natural candidate to replace the grid condition in the case of $\beta = -1$.
\end{rem}

It remains to prove the necessity of conditions \ref{Cond:fp} and \ref{Cond:tp} in Theorem \ref{thm:permuted_traffic}, which can be seen from the simplest nontrivial double tree.

\begin{lemma}\label{lem:necessary_conditions_traffic}
Suppose that $\mcal{M}_N$ converges in traffic distribution to a semicircular traffic family of covariance $\mbf{K}$ and pseudocovariance $\mbf{J}$. In that case, if $\beta \in (-1, 1)$, then
\begin{equation}\label{eq:fp_tp_limits_covariances}
\begin{aligned}
  \lim_{N \to \infty} \frac{\#\Big(\text{\emph{FP}}\big((\sigma_N^{(i')})^{-1} \circ \sigma_N^{(i)}\big)\Big)}{N^2} &= \frac{\mbf{K}(i, i') - \beta \mbf{J}(i, i')}{1 - \beta^2}; \\
  \lim_{N \to \infty} \frac{\#\Big(\text{\emph{TP}}\big((\sigma_N^{(i')})^{-1} \circ \sigma_N^{(i)}\big)\Big)}{N^2} &= \frac{-\beta\mbf{K}(i, i') + \mbf{J}(i, i')}{1 - \beta^2}.
\end{aligned}
\end{equation}
\end{lemma}
\begin{proof}
By assumption,
\[
  \begin{tikzpicture}
    \node at (-1.25, 0) {$\displaystyle \lim_{N \to \infty} \nu_{\mcal{M}_N}^0\Bigg[$};
    \node at (8, 0) {$\displaystyle \Bigg] = \lim_{N \to \infty} \frac{1}{N^2} \sum_{\phi: \{v_1, v_2\} \hookrightarrow [N]}  \E\big[\mbf{X}_N(\sigma_N^{(i)}(\phi(v_2), \phi(v_1)))\mbf{X}_N(\sigma_N^{(i')}(\phi(v_1), \phi(v_2)))\big] = \mbf{K}(i, i')$.};
    \draw[fill=black] (0,0) circle (1.75pt);
    \node at (0, -.375) {\footnotesize$v_1$\normalsize};
    \draw[fill=black] (1,0) circle (1.75pt);
    \node at (1, -.375) {\footnotesize$v_2$\normalsize};
    \draw[semithick, ->] (.125, .1) to node[midway, above] {\footnotesize$i$\normalsize} (.875, .1);
    \draw[semithick, ->] (.875, -.1) to node[midway, below] {\footnotesize$i'$\normalsize} (.125, -.1);
  \end{tikzpicture}  
\]
As before,
\begin{align*}
  &\sum_{\phi: \{v_1, v_2\} \hookrightarrow [N]} \E\big[\mbf{X}_N(\sigma_N^{(i)}(\phi(v_2), \phi(v_1)))\mbf{X}_N(\sigma_N^{(i')}(\phi(v_1), \phi(v_2)))\big] \\
  = &\sum_{\phi: \{v_1, v_2\} \hookrightarrow [N]}  \E\big[\mbf{X}_N((\sigma_N^{(i')})^{-1} \circ\sigma_N^{(i)}(\phi(v_2), \phi(v_1)))\mbf{X}_N(\phi(v_1), \phi(v_2))\big] \\
  = &\#\Big(\text{FP}\big((\sigma_N^{(i')})^{-1} \circ\sigma_N^{(i)}\big)\Big) + \beta\#\Big(\text{TP}\big((\sigma_N^{(i')})^{-1} \circ\sigma_N^{(i)}\big)\Big) + o(N^2).
\end{align*}
Thus,
\begin{equation}\label{eq:covariance_fp_tp}
  \lim_{N \to \infty} \frac{\#\Big(\text{FP}\big((\sigma_N^{(i')})^{-1} \circ\sigma_N^{(i)}\big)\Big) + \beta\#\Big(\text{TP}\big((\sigma_N^{(i')})^{-1} \circ\sigma_N^{(i)}\big)\Big)}{N^2} = \mbf{K}(i, i').
\end{equation}
A similar expansion of the limit
\[
  \begin{tikzpicture}
    \node at (-1.25, 0) {$\displaystyle \lim_{N \to \infty} \nu_{\mcal{M}_N}^0\Bigg[$};
    \node at (2.125, 0) {$\displaystyle \Bigg] = \mbf{J}(i, i')$};
    \draw[fill=black] (0,0) circle (1.75pt);
    \node at (0, -.375) {\footnotesize$v_1$\normalsize};
    \draw[fill=black] (1,0) circle (1.75pt);
    \node at (1, -.375) {\footnotesize$v_2$\normalsize};
    \draw[semithick, ->] (.875, .1) to node[midway, above] {\footnotesize$i$\normalsize} (.125, .1);
    \draw[semithick, ->] (.875, -.1) to node[midway, below] {\footnotesize$i'$\normalsize} (.125, -.1);
  \end{tikzpicture}  
\]
shows that
\begin{equation}\label{eq:pseudocovariance_fp_tp}
  \lim_{N \to \infty} \frac{\beta\#\Big(\text{FP}\big((\sigma_N^{(i')})^{-1} \circ\sigma_N^{(i)}\big)\Big) + \#\Big(\text{TP}\big((\sigma_N^{(i')})^{-1} \circ\sigma_N^{(i)}\big)\Big)}{N^2} = \mbf{J}(i, i').
\end{equation}
Since $\beta \in (-1, 1)$, the limits \eqref{eq:covariance_fp_tp} and \eqref{eq:pseudocovariance_fp_tp} are equivalent to the desired limits \eqref{eq:fp_tp_limits_covariances}.
\end{proof}

\begin{rem}\label{rem:fp_tp_limits_beta}
If $\beta \in \{-1, 1\}$, then either of the limits \eqref{eq:covariance_fp_tp} or \eqref{eq:pseudocovariance_fp_tp} is enough to determine the other, which explains the weaker condition \ref{Cond:fptp} replacing conditions \ref{Cond:fp} and \ref{Cond:tp} in Theorem \ref{thm:permuted_traffic}.
\end{rem}

Finally, if $\beta \in (-1, 1)$ (resp., $\beta = 1$), we prove that conditions \ref{Cond:fp} and \ref{Cond:tp} (resp. condition \ref{Cond:fptp}) with $a_{i, i'} = b_{i, i'} = 0$ for $i \neq i'$ and condition \ref{Cond:grid} are necessary and sufficient for the asymptotic traffic independence of the $(\mbf{M}_N^{(i)})_{i \in I}$. The necessity already follows from our work above and the definition of a semicircular traffic family (recall Example \ref{eg:semicircular_traffic_family}). To prove sufficiency, we need only to show that conditions \ref{Cond:fp} and \ref{Cond:tp} (resp. condition \ref{Cond:fptp}) with $a_{i, i'} = b_{i, i'} = 0$ are enough to repeat the calculation in Lemma \ref{lem:homogeneity_traffic}. In other words, we do not need the finer control within each row given by \eqref{eq:lemma_proportion_limits} if we know the overall proportion is vanishing.

\begin{lemma}\label{lem:vanishing_proportion_pure_double_trees}
Suppose that
\begin{equation}\label{eq:lemma_vanishing_proportion_limits}
  \lim_{N \to \infty} \frac{\#\Big(\text{\emph{FP}}\big((\sigma_N^{(i')})^{-1} \circ \sigma_N^{(i)}\big)\Big) + \#\Big(\text{\emph{TP}}\big((\sigma_N^{(i')})^{-1} \circ \sigma_N^{(i)}\big)\Big)}{N^2} = 0
\end{equation}
for any $i \neq i'$. In that case, if $T \in \mcal{T}\langle I \rangle$ is a double tree, then
\[
\begin{aligned}
  \lim_{N \to \infty} \nu_{\mcal{M}_N}^0[T] = \prod_{[e] = \{e, e'\} \in \wtilde{\mcal{N}}} \Big(&\indc{\{e, e'\} \text{ \emph{are opposing and }} \gamma(e) = \gamma(e')} \\
  + \beta&\indc{\{e, e'\} \text{ \emph{are congruent and }} \gamma(e) = \gamma(e')})\Big).
\end{aligned}
\]
\end{lemma}
\begin{proof}
Assume that $T = (V, E, \gamma) \in \mcal{T}\langle I \rangle$ is a double tree. First, suppose that $T$ has a mixed twin edge $[e_0] = \{e_0, e_0'\}$, i.e., $\gamma(e_0) \neq \gamma(e_0')$. We need to show that $\nu_{\mcal{M}_N}^0[T] = o_T(1)$. As in Lemma \ref{lem:homogeneity_traffic}, we can pass to a factorization at a negligible cost:
\[
  \nu_{\mcal{M}_N}^0[T] = \frac{1}{N^{\#(V)}}\sum_{\phi: V \hookrightarrow [N]} \prod_{[e] = \{e, e'\} \in \wtilde{\mcal{N}}} \E\Big[\mbf{X}_N(\sigma_N^{(\gamma(e))}(\phi(e)))\mbf{X}_N(\sigma_N^{(\gamma(e'))}(\phi(e')))\Big] + o_T(1).
\]
We rewrite the product to extract the mixed twin edge:
\[
  \E\Big[\mbf{X}_N(\sigma_N^{(\gamma(e_0))}(\phi(e_0)))\mbf{X}_N(\sigma_N^{(\gamma(e_0'))}(\phi(e_0')))\Big]\prod_{[e] = \{e, e'\} \in \wtilde{\mcal{N}}\setminus\{[e_0]\}} \E\Big[\mbf{X}_N(\sigma_N^{(\gamma(e))}(\phi(e)))\mbf{X}_N(\sigma_N^{(\gamma(e'))}(\phi(e')))\Big],
\]
where
\[
  \prod_{[e] = \{e, e'\} \in \wtilde{\mcal{N}}\setminus\{[e_0]\}} \E\Big[\mbf{X}_N(\sigma_N^{(\gamma(e))}(\phi(e)))\mbf{X}_N(\sigma_N^{(\gamma(e'))}(\phi(e')))\Big] = O_T(1)
\]
thanks to our strong moment assumption \eqref{eq:moment_bound}. This implies
\[
  \nu_{\mcal{M}_N}^0[T] = \frac{1}{N^{\#(V)}}\sum_{\phi: V \hookrightarrow [N]} \E\Big[\mbf{X}_N(\sigma_N^{(\gamma(e_0))}(\phi(e_0)))\mbf{X}_N(\sigma_N^{(\gamma(e_0'))}(\phi(e_0')))\Big]O_T(1) + o_T(1);
\]
however, the limit in \eqref{eq:lemma_vanishing_proportion_limits} implies that ``most'' maps $\phi: V \hookrightarrow [N]$ yield a zero summand
\[
  \E\Big[\mbf{X}_N(\sigma_N^{(\gamma(e_0))}(\phi(e_0)))\mbf{X}_N(\sigma_N^{(\gamma(e_0'))}(\phi(e_0')))\Big] = 0.
\]
In particular,
\[
  \sum_{\phi: V \hookrightarrow [N]} \E\Big[\mbf{X}_N(\sigma_N^{(\gamma(e_0))}(\phi(e_0)))\mbf{X}_N(\sigma_N^{(\gamma(e_0'))}(\phi(e_0')))\Big]O_T(1) = o_T(N^{\#(V)}),
\]
as was to be shown.

On the other hand, if $T$ does not have any mixed twin edges, then the limit follows from the factorization and the fact that
\[
  \E\Big[\mbf{X}_N(\sigma_N^{(\gamma(e))}(\phi(e)))\mbf{X}_N(\sigma_N^{(\gamma(e'))}(\phi(e')))\Big] = \indc{\{e, e'\} \text{ are opposing}} + \beta\indc{\{e, e'\} \text{ are congruent}}
\]
if $\gamma(e) = \gamma(e')$.
\end{proof}

\begin{rem}\label{rem:mixed_double_trees_exceptional_beta}
In the case of $\beta = -1$, the vanishing on mixed double trees would imply that
\[
  \begin{tikzpicture}
    \node at (-.75, 0) {$\displaystyle\nu_{\mcal{M}_N}^0\Bigg[$};
    \node at (3.0625, 0) {$\displaystyle\Bigg] = o_T(1)$};
    \draw[fill=black] (0,0) circle (1.75pt);
    \draw[fill=black] (1,0) circle (1.75pt);
    \draw[fill=black] (2,0) circle (1.75pt);
    \draw[semithick, ->] (.875, .0625) to node[midway, above] {\footnotesize$i$\normalsize} (.125, .0625);
    \draw[semithick, ->] (.125, -.0625) to node[midway, below] {\footnotesize$i'$\normalsize} (.875, -.0625);
    \draw[semithick, ->] (1.875, .0625) to node[midway, above] {\footnotesize$i'$\normalsize} (1.125, .0625);
    \draw[semithick, ->] (1.125, -.0625) to node[midway, below] {\footnotesize$i$\normalsize} (1.875, -.0625);
  \end{tikzpicture}  
\]
if $i \neq i'$. The same reasoning as in Remark \ref{rem:perfect_cancellation_exceptional_beta} allows us to conclude that
\[
  \sum_{j \in [N]} \left|\#\Big(\text{FP}\big((\sigma_N^{(i')})^{-1} \circ \sigma_N^{(i)}, j\big)\Big) - \#\Big(\text{TP}\big((\sigma_N^{(i')})^{-1} \circ \sigma_N^{(i)}, j\big)\Big)\right| = o_{i, i'}(N^2).
\]
This seems like the most natural candidate to replace \eqref{eq:lemma_vanishing_proportion_limits} in the case of $\beta = -1$.
\end{rem}

\subsection{Asymptotics for the distribution of $(\mbf{M}_N^{(i)})_{i \in I}$}\label{sec:joint_distribution}
We now restrict our attention to the distribution of $(\mbf{M}_N^{(i)})_{i \in I}$. Equation \eqref{eq:traffic_to_ncps} allows us to translate this problem to understanding the injective traffic distribution of $(\mbf{M}_N^{(i)})_{i \in I}$ on test graphs that can be obtained as a quotient of a simple directed cycle. Note that Lemma \ref{lem:double_tree} still applies, allowing us to restrict our attention to quotients that are double trees. Since a double tree obtained as a quotient of a simple directed cycle can only have opposing twin edges \cite[Figure 5]{Au18a}, condition \ref{ncond:fp_tp} of Corollary \ref{cor:permuted_nc} is already enough to calculate the limit in Equation \eqref{eq:stripping_double_tree_limit} of Lemma \ref{lem:homogeneity_traffic}. This proves the first part of Corollary \ref{cor:permuted_nc}.

If $\beta = 0$, then we do not need to worry about the full grid in Equation \eqref{eq:adjacent_singleton_matching}. Indeed, if $[e]$ and $[e']$ are adjacent singleton edges matched by $\phi$, then either
  \[
    \begin{tikzpicture}[baseline=(current  bounding  box.center), shorten > = 2.5pt]
      \draw[fill=black] (2, 0) circle (1.5pt);
      \draw[fill=black] (3, 0) circle (1.5pt);
      \draw[fill=black] (4, 0) circle (1.5pt);

      \draw[semithick, ->] (3, 0) to node[pos=.425,above] {\footnotesize$e$\normalsize} (2, 0);
      \draw[semithick, ->] (4, 0) to node[pos=.425,above] {\footnotesize$e'$\normalsize} (3, 0);
   \end{tikzpicture}
 \]
or
 \[
    \begin{tikzpicture}[baseline=(current  bounding  box.center), shorten > = 2.5pt]
      \draw[fill=black] (2, 0) circle (1.5pt);
      \draw[fill=black] (3, 0) circle (1.5pt);
      \draw[fill=black] (4, 0) circle (1.5pt);

      \draw[semithick, ->] (2, 0) to node[pos=.425,above] {\footnotesize$e$\normalsize} (3, 0);
      \draw[semithick, ->] (3, 0) to node[pos=.425,above] {\footnotesize$e'$\normalsize} (4, 0);
   \end{tikzpicture}
 \]
since the subsequent pinching must produce an opposing twin edge. We can assume that we are in the first case without loss of generality. Then either
\[
  (\sigma_N^{\gamma(e')})^{-1} \circ \sigma_N^{\gamma(e)}(\phi(\tar(e)), \phi(\src(e))) = (\phi(\tar(e')), \phi(\src(e')))
\]
or
\[
  (\sigma_N^{\gamma(e')})^{-1} \circ \sigma_N^{\gamma(e)}(\phi(\tar(e)), \phi(\src(e))) = (\phi(\src(e')), \phi(\tar(e'))).
\]
The first case produces a zero term
\[
  \E\big[\mbf{X}_N(\sigma_N^{(\gamma(e))}(\phi(e)))\mbf{X}_N(\sigma_N^{(\gamma(e'))}(\phi(e')))\big] = \beta = 0,
\]
while the second case is covered by Popa's condition \eqref{eq:popa_condition}.

If one is only interested in freeness, then we can settle for the averaged condition in Equation \ref{ncond:freeness}, which can be used to replace the condition in Equation \eqref{eq:lemma_vanishing_proportion_limits} of Lemma \ref{lem:vanishing_proportion_pure_double_trees} (recall that we only care about double trees with opposing twin edges). A similar argument proves the sufficiency of condition \ref{ncond:general_beta} for general $\beta \in \mathbb{D}$. This completes the proof of Corollary \ref{cor:permuted_nc}.

\addtocontents{toc}{\SkipTocEntry}
\subsection*{Acknowledgments} The author thanks Jonathan Novak for suggesting an investigation into the joint distribution of Wigner matrices with permuted entries. The simulations in this article were performed in Julia \cite{BEKS17} and the data plotted using Gadfly \cite{Gadfly}. The figures in this article were produced in Inkscape.

\bibliographystyle{amsalpha}
\bibliography{master_bib}

\end{document}

%% file: fig_inhomogeneous.pdf_tex
\begingroup%
  \makeatletter%
  \providecommand\color[2][]{%
    \errmessage{(Inkscape) Color is used for the text in Inkscape, but the package 'color.sty' is not loaded}%
    \renewcommand\color[2][]{}%
  }%
  \providecommand\transparent[1]{%
    \errmessage{(Inkscape) Transparency is used (non-zero) for the text in Inkscape, but the package 'transparent.sty' is not loaded}%
    \renewcommand\transparent[1]{}%
  }%
  \providecommand\rotatebox[2]{#2}%
  \newcommand*\fsize{\dimexpr\f@size pt\relax}%
  \newcommand*\lineheight[1]{\fontsize{\fsize}{#1\fsize}\selectfont}%
  \ifx\svgwidth\undefined%
    \setlength{\unitlength}{360bp}%
    \ifx\svgscale\undefined%
      \relax%
    \else%
      \setlength{\unitlength}{\unitlength * \real{\svgscale}}%
    \fi%
  \else%
    \setlength{\unitlength}{\svgwidth}%
  \fi%
  \global\let\svgwidth\undefined%
  \global\let\svgscale\undefined%
  \makeatother%
  \begin{picture}(1,0.25)%
    \lineheight{1}%
    \setlength\tabcolsep{0pt}%
    \put(0.00822722,1.97633939){\color[rgb]{0,0,0}\makebox(0,0)[lt]{\begin{minipage}{0.14438897\unitlength}\raggedright \end{minipage}}}%
    \put(0,0){\includegraphics[width=\unitlength,page=1]{fig_inhomogeneous.pdf}}%
  \end{picture}%
\endgroup%

%% file: fig_grid.pdf_tex
\begingroup%
  \makeatletter%
  \providecommand\color[2][]{%
    \errmessage{(Inkscape) Color is used for the text in Inkscape, but the package 'color.sty' is not loaded}%
    \renewcommand\color[2][]{}%
  }%
  \providecommand\transparent[1]{%
    \errmessage{(Inkscape) Transparency is used (non-zero) for the text in Inkscape, but the package 'transparent.sty' is not loaded}%
    \renewcommand\transparent[1]{}%
  }%
  \providecommand\rotatebox[2]{#2}%
  \newcommand*\fsize{\dimexpr\f@size pt\relax}%
  \newcommand*\lineheight[1]{\fontsize{\fsize}{#1\fsize}\selectfont}%
  \ifx\svgwidth\undefined%
    \setlength{\unitlength}{360bp}%
    \ifx\svgscale\undefined%
      \relax%
    \else%
      \setlength{\unitlength}{\unitlength * \real{\svgscale}}%
    \fi%
  \else%
    \setlength{\unitlength}{\svgwidth}%
  \fi%
  \global\let\svgwidth\undefined%
  \global\let\svgscale\undefined%
  \makeatother%
  \begin{picture}(1,0.25)%
    \lineheight{1}%
    \setlength\tabcolsep{0pt}%
    \put(0.00822722,1.97633939){\color[rgb]{0,0,0}\makebox(0,0)[lt]{\begin{minipage}{0.14438897\unitlength}\raggedright \end{minipage}}}%
    \put(0,0){\includegraphics[width=\unitlength,page=1]{fig_grid.pdf}}%
  \end{picture}%
\endgroup%

%% file: fig_cancellation.pdf_tex
\begingroup%
  \makeatletter%
  \providecommand\color[2][]{%
    \errmessage{(Inkscape) Color is used for the text in Inkscape, but the package 'color.sty' is not loaded}%
    \renewcommand\color[2][]{}%
  }%
  \providecommand\transparent[1]{%
    \errmessage{(Inkscape) Transparency is used (non-zero) for the text in Inkscape, but the package 'transparent.sty' is not loaded}%
    \renewcommand\transparent[1]{}%
  }%
  \providecommand\rotatebox[2]{#2}%
  \newcommand*\fsize{\dimexpr\f@size pt\relax}%
  \newcommand*\lineheight[1]{\fontsize{\fsize}{#1\fsize}\selectfont}%
  \ifx\svgwidth\undefined%
    \setlength{\unitlength}{360bp}%
    \ifx\svgscale\undefined%
      \relax%
    \else%
      \setlength{\unitlength}{\unitlength * \real{\svgscale}}%
    \fi%
  \else%
    \setlength{\unitlength}{\svgwidth}%
  \fi%
  \global\let\svgwidth\undefined%
  \global\let\svgscale\undefined%
  \makeatother%
  \begin{picture}(1,0.6)%
    \lineheight{1}%
    \setlength\tabcolsep{0pt}%
    \put(0.12000119,0.0136559){\makebox(0,0)[t]{\lineheight{1.25}\smash{\begin{tabular}[t]{c}-4\end{tabular}}}}%
    \put(0.31913505,0.0136559){\makebox(0,0)[t]{\lineheight{1.25}\smash{\begin{tabular}[t]{c}-2\end{tabular}}}}%
    \put(0.51826891,0.0136559){\makebox(0,0)[t]{\lineheight{1.25}\smash{\begin{tabular}[t]{c}0\end{tabular}}}}%
    \put(0.71748151,0.0136559){\makebox(0,0)[t]{\lineheight{1.25}\smash{\begin{tabular}[t]{c}2\end{tabular}}}}%
    \put(0.91661536,0.0136559){\makebox(0,0)[t]{\lineheight{1.25}\smash{\begin{tabular}[t]{c}4\end{tabular}}}}%
    \put(0,0){\includegraphics[width=\unitlength,page=1]{fig_cancellation.pdf}}%
    \put(0.09637914,0.04841721){\makebox(0,0)[rt]{\lineheight{1.25}\smash{\begin{tabular}[t]{r}0.0\end{tabular}}}}%
    \put(0.09637914,0.16770855){\makebox(0,0)[rt]{\lineheight{1.25}\smash{\begin{tabular}[t]{r}0.1\end{tabular}}}}%
    \put(0.09637914,0.28699989){\makebox(0,0)[rt]{\lineheight{1.25}\smash{\begin{tabular}[t]{r}0.2\end{tabular}}}}%
    \put(0.09637914,0.40636997){\makebox(0,0)[rt]{\lineheight{1.25}\smash{\begin{tabular}[t]{r}0.3\end{tabular}}}}%
    \put(0.09637914,0.52566131){\makebox(0,0)[rt]{\lineheight{1.25}\smash{\begin{tabular}[t]{r}0.4\end{tabular}}}}%
    \put(0.73572695,0.46833187){\makebox(0,0)[lt]{\lineheight{1.25}\smash{\begin{tabular}[t]{l}{\small $(\beta, n) = (-\frac{1}{2}, 3)$}\end{tabular}}}}%
    \put(0,0){\includegraphics[width=\unitlength,page=2]{fig_cancellation.pdf}}%
    \put(0.7354827,0.49815479){\makebox(0,0)[lt]{\lineheight{1.25}\smash{\begin{tabular}[t]{l}{\small $(\beta, n) = (-1, 2)$}\end{tabular}}}}%
    \put(0,0){\includegraphics[width=\unitlength,page=3]{fig_cancellation.pdf}}%
    \put(0.73572645,0.43850896){\makebox(0,0)[lt]{\lineheight{1.25}\smash{\begin{tabular}[t]{l}{\small $\nu_{\mcal{SP}}(dx)$}\end{tabular}}}}%
  \end{picture}%
\endgroup%

%% file: fig_anti_transpose.pdf_tex
\begingroup%
  \makeatletter%
  \providecommand\color[2][]{%
    \errmessage{(Inkscape) Color is used for the text in Inkscape, but the package 'color.sty' is not loaded}%
    \renewcommand\color[2][]{}%
  }%
  \providecommand\transparent[1]{%
    \errmessage{(Inkscape) Transparency is used (non-zero) for the text in Inkscape, but the package 'transparent.sty' is not loaded}%
    \renewcommand\transparent[1]{}%
  }%
  \providecommand\rotatebox[2]{#2}%
  \newcommand*\fsize{\dimexpr\f@size pt\relax}%
  \newcommand*\lineheight[1]{\fontsize{\fsize}{#1\fsize}\selectfont}%
  \ifx\svgwidth\undefined%
    \setlength{\unitlength}{360bp}%
    \ifx\svgscale\undefined%
      \relax%
    \else%
      \setlength{\unitlength}{\unitlength * \real{\svgscale}}%
    \fi%
  \else%
    \setlength{\unitlength}{\svgwidth}%
  \fi%
  \global\let\svgwidth\undefined%
  \global\let\svgscale\undefined%
  \makeatother%
  \begin{picture}(1,0.6)%
    \lineheight{1}%
    \setlength\tabcolsep{0pt}%
    \put(0.12000204,0.01365591){\makebox(0,0)[t]{\lineheight{1.25}\smash{\begin{tabular}[t]{c}-4\end{tabular}}}}%
    \put(0.3191359,0.01365591){\makebox(0,0)[t]{\lineheight{1.25}\smash{\begin{tabular}[t]{c}-2\end{tabular}}}}%
    \put(0.51826976,0.01365591){\makebox(0,0)[t]{\lineheight{1.25}\smash{\begin{tabular}[t]{c}0\end{tabular}}}}%
    \put(0.71748236,0.01365591){\makebox(0,0)[t]{\lineheight{1.25}\smash{\begin{tabular}[t]{c}2\end{tabular}}}}%
    \put(0.91661621,0.01365591){\makebox(0,0)[t]{\lineheight{1.25}\smash{\begin{tabular}[t]{c}4\end{tabular}}}}%
    \put(0,0){\includegraphics[width=\unitlength,page=1]{fig_anti_transpose.pdf}}%
    \put(0.09637989,0.04841722){\makebox(0,0)[rt]{\lineheight{1.25}\smash{\begin{tabular}[t]{r}0.0\end{tabular}}}}%
    \put(0.09637989,0.16770856){\makebox(0,0)[rt]{\lineheight{1.25}\smash{\begin{tabular}[t]{r}0.1\end{tabular}}}}%
    \put(0.09637989,0.2869999){\makebox(0,0)[rt]{\lineheight{1.25}\smash{\begin{tabular}[t]{r}0.2\end{tabular}}}}%
    \put(0.09637989,0.40636998){\makebox(0,0)[rt]{\lineheight{1.25}\smash{\begin{tabular}[t]{r}0.3\end{tabular}}}}%
    \put(0.09637989,0.52566132){\makebox(0,0)[rt]{\lineheight{1.25}\smash{\begin{tabular}[t]{r}0.4\end{tabular}}}}%
    \put(0.7878112,0.46833239){\makebox(0,0)[lt]{\lineheight{1.25}\smash{\begin{tabular}[t]{l}{\small $\beta = \sqrt{-1}$}\end{tabular}}}}%
    \put(0,0){\includegraphics[width=\unitlength,page=2]{fig_anti_transpose.pdf}}%
    \put(0.78756695,0.4981553){\makebox(0,0)[lt]{\lineheight{1.25}\smash{\begin{tabular}[t]{l}{\small $\beta = 1$}\end{tabular}}}}%
    \put(0,0){\includegraphics[width=\unitlength,page=3]{fig_anti_transpose.pdf}}%
    \put(0.78781065,0.43850949){\makebox(0,0)[lt]{\lineheight{1.25}\smash{\begin{tabular}[t]{l}{\small $\nu_{\mcal{SP}}(dx)$}\end{tabular}}}}%
  \end{picture}%
\endgroup%